\documentclass{elsarticle}
\usepackage{times}
\usepackage{amsthm}
\usepackage{url}
\usepackage{mathtools}
\usepackage{setspace}
\usepackage{fullpage}

\usepackage{fancyhdr}
\usepackage{dsfont}
\usepackage{subfig}
\usepackage{setspace}
\newtheorem{thm}{Theorem}
\newtheorem{lem}{Lemma}
\newtheorem*{mydef}{Definition}
\newtheorem{case}{Case}
\newtheorem*{corlem}{Corollary}
\newtheorem{cor}{Corollary}
\newtheorem*{obs}{Observation}
\newtheorem*{lem1}{Corner Lemma}
\newtheorem*{lem2}{Merging Lemma}
\newtheorem*{lem3}{Large Cut Set Lemma}
\newcommand{\tria}[1]{\text{tri}_{#1}}
\newcommand{\tri}[1]{\text{tri}\left(#1\right)}
\newcommand{\sq}[1]{G_{#1,#1}}
\newcommand{\floor}[1]{\left\lfloor #1\right\rfloor}
\newcommand{\ceil}[1]{\left\lceil #1\right\rceil}
\title{
On the Rank Number of Grid Graphs
}

\author{
Sitan Chen \fnref{sc}\\
\emph{Massachusetts Institute of Technology, Cambridge, MA 02139}\\
\url{sitanchen@college.harvard.edu}}

\fntext[sc]{Present Address: Harvard College, Cambridge, MA 02138}
\date{August 9, 2012}

\begin{document}

\pagestyle{fancy}
\lhead{}
\chead{}
\rhead{}
\lfoot{}
\cfoot{}
\rfoot{}
\renewcommand{\headrulewidth}{0mm}

\thispagestyle{fancyplain}
\begin{frontmatter}
\begin{abstract}
A vertex $k$-ranking is a labeling of the vertices of a graph with integers from 1 to $k$ so any path connecting two vertices with the same label will pass through a vertex with a greater label.  The \emph{rank number} of a graph is defined to be the minimum possible $k$ for which a $k$-ranking exists for that graph.  For $m\times n$ grid graphs, the rank number has been found only for $m\le 3$.  In this paper, we determine its for $m=4$ and improve its upper bound for general grids.  Furthermore, we improve lower bounds on the rank numbers for square and triangle grid graphs from logarithmic to linear.  These new lower bounds are key to characterizing the rank number for general grids, and our results have applications in optimizing VLSI circuit design and parallel processing, search, and scheduling.
\end{abstract}
\end{frontmatter}

\setstretch{1}
\pagestyle{fancyplain}

\section{Introduction}

Given a graph $G$, a \emph{graph labeling} is a function that takes the vertices $V(G)$ in a graph $G$ to a subset of the integers, subject to certain constraints.  One notable graph labeling, a $k$-coloring, labels a graph with $k$ colors so no two adjacent vertices share the same color \cite{kcoloring}.

The very large-scale integration (VLSI) circuit layout motivated the study of $k$-ranking, which can be thought of as a generalization of $k$-coloring.  A VLSI circuit consists of a large number of transistors and wires contained within a multi-layer chip.  If we treat transistors as vertices and wires as edges in a graph, then many graph properties are related to circuit features \cite{embedding,testability}.  In particular, $k$-ranking was first studied for its connections to finding minimal separators of graphs \cite{node,trees}.  Finding minimal separators is key to minimizing VLSI layout area, which is directly related to hardware expense \cite{vlsi}.

In this paper, we first introduce some definitions and previous results in Section 2. Then in Section 3, we determine the closed form for the rank number of $G_{4,n}$, and in Section 4, we improve previous upper bounds for $\chi_r(G_{m,n})$.  In \cite{alpert}, Alpert called for a strong lower bound for $\chi_r(\sq{n})$, which is key to completely determining $\chi_r(G_{m,n})$, so in Section 5, we present a linear lower bound.  As a corollary, this also gives a linear lower bound for the rank number of triangle grids $\tria{n}$ (see Figure~\ref{fig:tri}).

The results on $\chi_r(G_{4,n})$ and $\chi_r(G_{m,n})$ and the new methods used to obtain them offer insight on how to approach completely determining the rank number of grid graphs given a suitable lower bound on the rank number of square grid graphs, one that we provide.  Combined, the results of this research provide substantial groundwork for solving the open problem of determining $\chi_r(G_{m,n})$.  

Furthermore, the novel methods we introduce to study minimal separators in obtaining an improved lower bound on $\chi_r(\sq{n})$ are key not only to optimizing VLSI circuit design, but also to optimizing many parallel algorithms, including parallel scheduling of multi-part product assembly in manufacturing systems, searching for corruptions in partially ordered data structures, parallel query processing, and Cholesky factorizing matrices in parallel \cite{parallel}.

\section{Fundamentals}

Gallian \cite{survey} gives the following definition.

\begin{mydef}
A labeling function $f: V(G)\to\{1,...,k\}$ is a \emph{$k$-ranking} of a graph $G$ if each path between two vertices of the same value passes through a vertex with a larger value.  The \emph{rank number} $\chi_r(G)$ is the smallest $k$ for which a $k$-ranking of $G$ exists.  Furthermore, a $k$-ranking is \emph{minimal} if no label can be replaced with a smaller value and still satisfy the conditions for a ranking.
\end{mydef}

Note this is a generalization of $k$-coloring, which only involves paths of length 1, to paths of arbitrary length.  We can make several self-evident observations.

\begin{lem}
For a connected graph, the largest value $k$ used in its $k$-ranking is unique.
\end{lem}

\begin{lem}
If $H$ is a subgraph of $G$, then $\chi_r(H)\le\chi_r(G)$.  
\end{lem}

Lemmas 1 and 2 imply that the rank number monotonically increases with the addition of vertices \cite{kostyuk}.

\begin{lem}
For any vertex $v\in G$, $\chi_r(G)-1\le \chi_r(G-v)\le \chi_r(G)$.
\end{lem}

\begin{figure}[h]
 \begin{center}
 \subfloat[$P_5$ or $G_{1,5}$]{\label{fig:path}{\includegraphics[width=3cm]{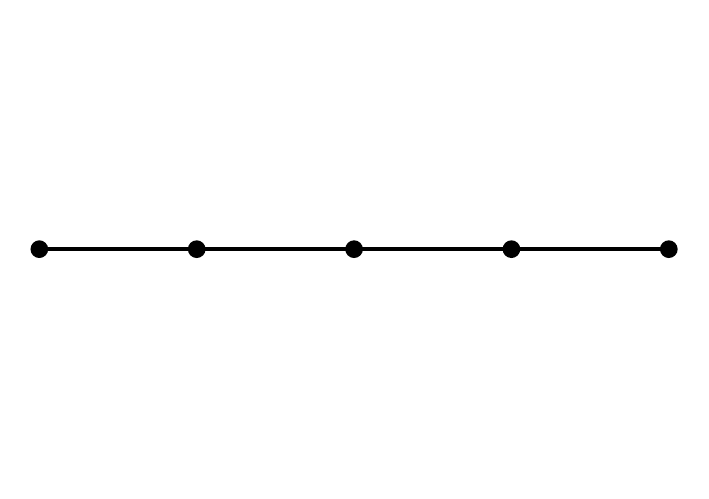}}}
\hspace{0.3cm}
 \subfloat[$G_{4,5}$]{\includegraphics[angle=90,height=1.9cm]{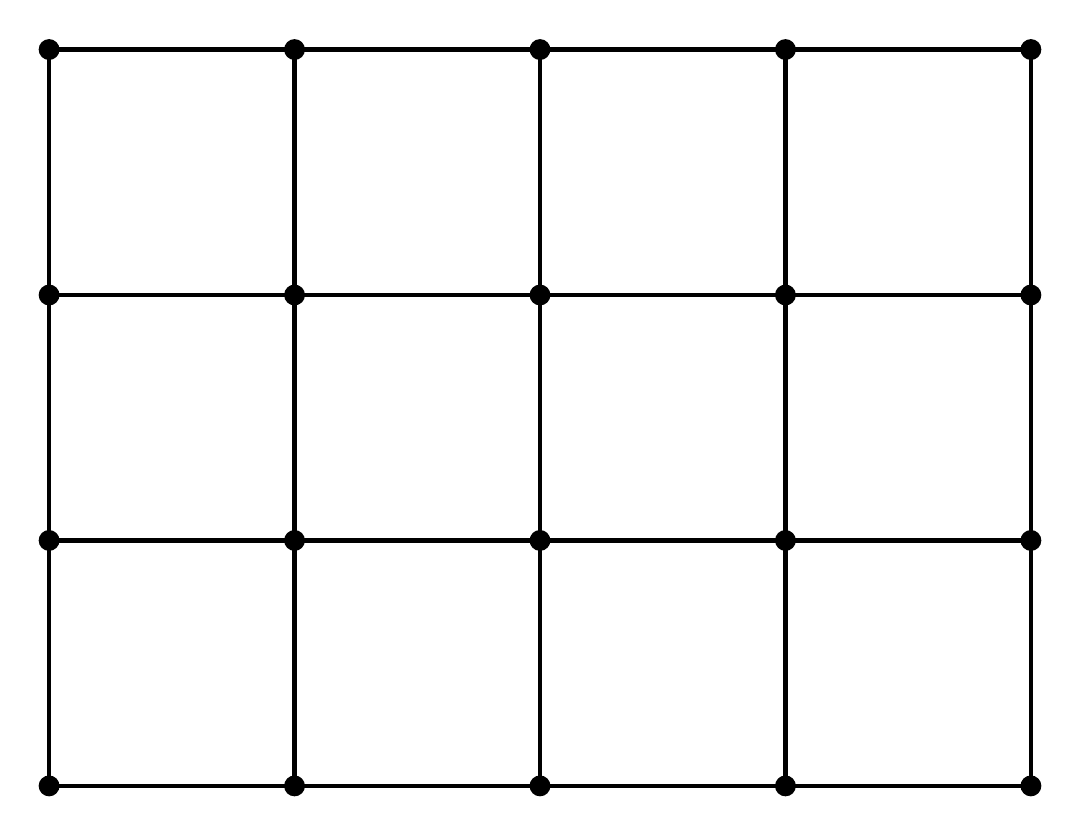}}
\hspace{0.3cm}
 \subfloat[$\tria{5}$]{\label{fig:tri}\includegraphics[height=2cm]{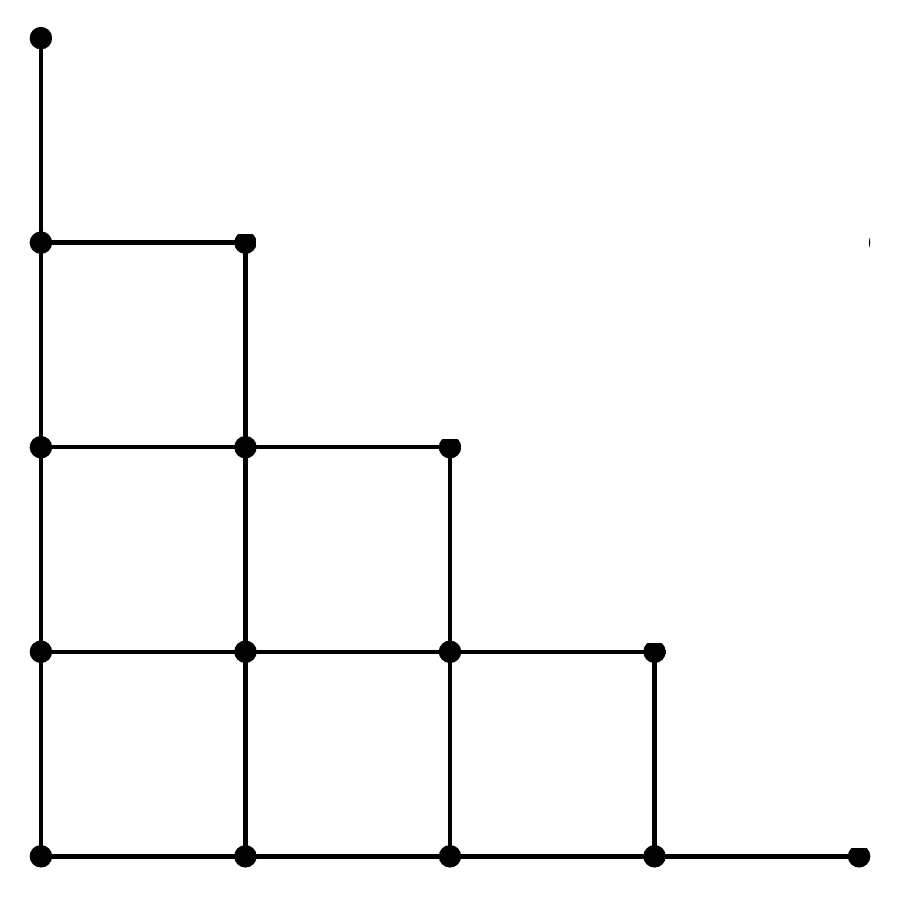}}
 \caption{Notation for path, grid, and triangle graphs.}
 \end{center}
\end{figure}

For simplicity, we denote the rank number of $G_{m,n}$ by $r(m,n)$. According to Gallian \cite{survey}, the rank number of a path graph (Figure~\ref{fig:path}) is known to be $\chi_r(P_n)=r(1,n)=\floor{\log_2(n)}+1=1+\chi_r(P_{\ceil{\frac{n-1}{2}}})$. Chang et al. \cite{chang} showed that $r(2,n)=2+r(2,\floor{\frac{n-2}{2}})$;  Alpert \cite{alpert} showed that
\begin{equation*}
r(3,n)= 
\begin{dcases}
4+r\left(3,\ceil{\frac{n-3}{2}}\right) &\text{for} \ n=15\cdot4^k+7\cdot\frac{4^k-1}{3}+\frac{3\pm1}{2}\\
3+r\left(3,\ceil{\frac{n-3}{2}}\right) &\text{for} \ n\neq15\cdot4^k+7\cdot\frac{4^k-1}{3}+\frac{3\pm1}{2}
\end{dcases}
\end{equation*}  

The rank number of a grid graph $G_{m,n}$, however, is an open problem for $m\ge 4$.  

\section{Rank Number of $G_{4,n}$}

Let $b(n)=2b_2+b_3$, where $b_i$ is the $i$th most significant bit of $n$.

\begin{thm}
For $n>8$ if $n=2^k+2^{k-2}-2$ or $2^k+2^{k-2}-1$, then $r(n)=4k-2$.  Otherwise, \begin{equation*}r(4,n)=4\floor{\log_2(n+1)}-3+b(n+1).\end{equation*}

Alternatively, if $I=[2^k-1,2^k]\cup\{2^k+2^{k-1}-2\}\cup[2^k+2^{k-1}+2^{k-2}-1,2^k+2^{k-1}+2^{k-2}]$, then \begin{equation*}r(4,n)= 
\begin{dcases}
	5+r\left(4,\ceil{\frac{n-4}{2}}\right) & \ \emph{if} \  n \in I\\
	4+r\left(4,\ceil{\frac{n-4}{2}}\right) & \ \emph{if} \  n \not\in I.
\end{dcases}
\end{equation*}\end{thm}

Before we prove this, we define some terms.  Let $\alpha$ be the largest value that is used more than once in a given minimal ranking $f$ of $G_{4,n}$ and $T_{\alpha}$ the set of vertices labeled with values greater than $\alpha$.  Let the \emph{cut set} $C(f)$ be a minimal subset of vertices in $T_{\alpha}$ so that the vertices labeled $\alpha$ are in different connected components of $G\backslash C(f)$, and let $|C(f)|$ be the number of vertices in $C(f)$.

We first find tight lower and upper bounds for the rank number of $4\times n$ grids.  

\subsection{Lower Bound}

\begin{lem}
For $n>5$, \begin{equation*}r(4,n)\ge 4+r\left(4,\ceil{\frac{n-4}{2}}\right).\end{equation*}
\end{lem}

\begin{proof}
For $G_{4,n}$, with $n>1$, clearly $|C(f)|>1$.  If $|C(f)|=2$, then the cut set must be the two neighbors of a corner, and $\alpha$ is the value of the corner.  Because the ranking is minimal, $\alpha=1$, but if $n>2$, then $\alpha>1$.  Similarly, if $|C(f)|=3$, then it is easy to check that any configuration of cut set vertices will split the graph into two connected components, one of which must have a rank number of at most 2.  Again, because the ranking is minimal, $\alpha\le 2$, but if $n>3$ then $\alpha>2$.  Thus, for $n>3$, we have that $|C(f)|\ge 4$.  In the cut set, any four vertices will occupy at most four consecutive columns, and $r(4,n)\ge 4+r\left(4,\ceil{\frac{n-4}{2}}\right),$ because the complement of those four vertices will contain $G_{4,\ceil{\frac{n-4}{2}}}$. \end{proof}

\subsection{Upper Bound}

\begin{figure}[h]
\begin{center}
\includegraphics[width=\textwidth]{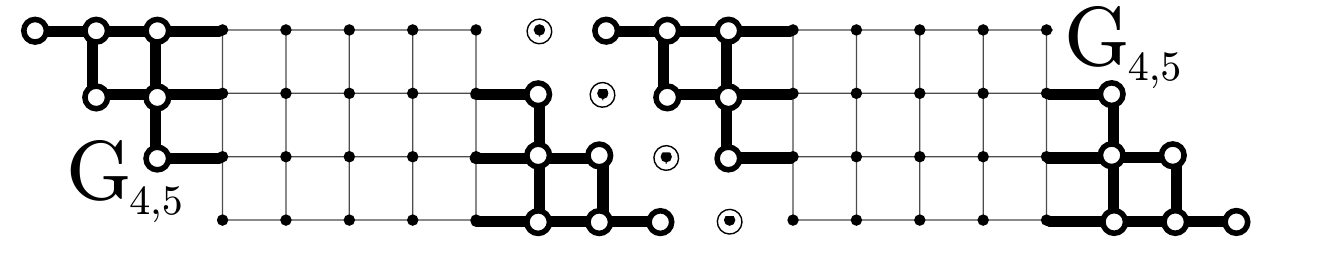}
\caption{Merging two $4\times 5$ grids with sticky ends (bolded) to get a $4\times 14$ grid with sticky ends.}
\label{fig:sticky}
\end{center}
\end{figure}

We next find the upper bound.  Define a \emph{sticky end} to be the staircase-like configuration attached to the end of a grid, bolded in Figure~\ref{fig:sticky}.

\begin{lem}
If $G_{4,n}$ with a sticky end on both sides can be labeled with $\lambda$ colors, then \begin{equation*}r(4,2n+4)\le \lambda+4.\end{equation*}
\end{lem}

\begin{proof}Place two such $G_{4,n}$ graphs with sticky ends, both colored in the same way, together and let the cut set be four vertices between these two graphs, colored with $\lambda+1$, $\lambda+2$, $\lambda+3$, and $\lambda+4$.  This gives a $(\lambda+4)$-ranking of $G_{4,2n+4}$ with two sticky ends.  Because $G_{4,2n+4}$ with no sticky ends is its subgraph, $r(4,2n+4)\le \lambda+4$.\end{proof}

Figure~\ref{fig:sticky} shows an example for constructing $G_{4,14}$ with two sticky ends out of two copies of $G_{4,5}$ with two sticky ends.


In \cite{alpert}, Alpert showed that $r(4,3)=6$.  We show in \ref{app:basecase} that $r(4,4)=7$, $r(4,5)=8$, $r(4,6)=8$, $r(4,7)=9$, and $r(4,8)=10$.  Define $s(4,n)$ to be the rank number of $G_{4,n}$ with two sticky ends.  We also show by construction in \ref{app:basesticky} that $s(4,1)\le5$, $s(4,2)\le6$, $s(4,3)\le7$, $s(4,4)\le8$.


We rewrite the lower bound as $r(4,2n+3)\ge 4+r(4,n)$ and the upper bound as $r(4,2n+4)\le s(2n+4)\le 4+s(n)$ and combine these bounds so that for each $n$, $r(4,n)$ is bounded between two consecutive integers.  We can condense these combined bounds to get that for $k\ge 3$, where $k=\floor{\log_2(n+3)}$,

\begin{equation}4k-4+i\le r(4,n)\le 4k-3+i \hspace{0.5cm} \text{for} \hspace{0.5cm} n\in[B_{k,i},B_{k,i+1}),\end{equation}\label{eq:bounds} where $0\le i\le 3$ and $B_{k,0}=2^k-3$, $B_{k,1}=2^k+2^{k-2}-3$, $B_{k,2}=2^k+2^{k-1}-3$, $B_{k,3}=2^k+2^{k-1}+2^{k-2}-3$, a nd $B_{k,4}=2^{k+1}-3$.

\subsection{Completely Characterizing $r(4,n)$}


We restrict to cases where $n>8$ and $|C(f)|=4$. Note that the latter preserves generality for the arguments following this lemma because given a cut set of four vertices and the corresponding widest possible induced subgrid, an addition of $k$ extra vertices to the cut set will occupy at most $k$ vertices of that subgrid. Thus, the rank number of the subgrid with these vertices deleted is at most $k$ less than the original subgrid's rank number, while the addition of $k$ vertices to the cut set increases the rank number by $k$. This, combined with the result in section 3.1 that $|C(f)|\ge4$, lets us restrict our attention to grids of exactly four cut set vertices for the rest of Section 3.

In each interval $[B_{k,i},B_{k,i+1})$, we aim to find the closed form of $n$ such that $r(4,n+1)>r(4,n)$. 

We first prove two lemmas. The first, the Merging Lemma, will provide a construction that we will show for each integer $0\le i\le 3$ guarantees the existence of a $(4k-4+i)$-ranking for all $G_{4,n_1}$, where $n_1\in[B_{k,i},M]$ for some integer $M\le B_{k,i+1}$. The second, the Corner Lemma, will be used to show there only exists a $(4k-3+i)$-ranking for all $G_{4,n_2}$, where $n_2\in(M,B_{k,i+1})$.

\begin{figure}[h]
 \begin{center}
  \includegraphics[width=\textwidth]{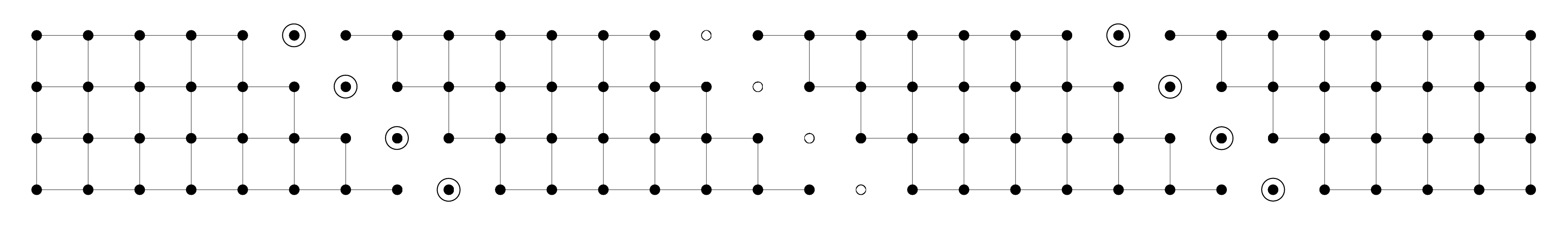}
 \end{center}
 \caption{$G_{4,2^5-2}$ split by 4 high vertices (white) into two copies of $G_{4,2^4-3}$, which are each split by 4 high vertices (circled) into $G_{2^{3}-3}$ with a sticky end and $G_{2^{3}-4}$ with two sticky ends}
 \label{fig:split}
\end{figure}

\begin{lem2} If $G_{4,n}$ with a sticky end has a $\lambda$-ranking and $G_{4,n-1}$ with two sticky ends has a $\lambda$-ranking, then $G_{4,2n+3}$ with a sticky end has a $(\lambda+4)$-ranking and $G_{4,4n+10}$ has a $(\lambda+8)$-ranking.\end{lem2}

\begin{proof} If $G_{4n+10}$ is bisected with a cut set of 4 high vertices descending diagonally from left to right in the middle, then the induced subgraphs are two copies of $G_{4,2n+3}$ each with one sticky end.  Each induced subgraph is further bisected with a cut set of 4 high vertices parallel to the sticky end to get a $G_{4,n}$ subgraph with a sticky end and a $G_{4,n-1}$ subgraph with two sticky ends.  Both have $\lambda$-rankings, so $G_{4,2n+3}$ with a sticky end has a $(\lambda+4)$-ranking and $G_{4,4n+10}$ has a $(\lambda+8)$-ranking.\end{proof}

This division of $G_{4,4n+10}$ is shown for $n=5$ in Figure~\ref{fig:split}. 

Next, define a \emph{grid graph $G$ missing a corner} to be $G \backslash v$, where $v$ is a corner vertex.

\begin{lem1}One of the two connected components induced from removing the cut set from $G_{4,n}$ contains the subgraph $G_{4,\ceil{\frac{n-2}{2}}}$ missing one corner.\end{lem1}

\begin{proof}

It is easy to see that this is true if the cut set occupies two or fewer consecutive columns.  If the cut set occupies three columns, then at least one of the vertices in the cut set on the top or bottom row will not share a column with the rest of the cut set.  If we remove the other three points in two columns, we induce a $G_{4,\ceil{\frac{n-2}{2}}}$ subgraph, and by removing the remaining point in the cut set, we remove its corner.

\begin{figure}[h]
 \begin{center}
  \includegraphics[scale=0.4]{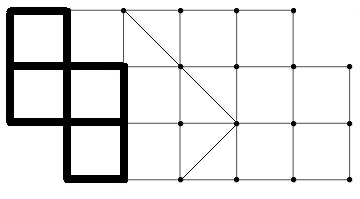}
 \end{center}
  \caption{Removing any cut set of $G_{4,7}$ missing two corners in different columns induces a subgraph of $G_{4,3}$ minus two corners in different columns}
  \label{fig:cutcorner}
\end{figure}
 
If the cut set occupies four columns, then we can similarly remove the middle two columns to induce a $G_{4,\ceil{\frac{n-2}{2}}}$ subgraph, and by removing either of the two remaining points, we remove its corner.  
\end{proof}

\begin{corlem} For $n\ge 5$, if the rank number of $G_{4,\ceil{\frac{n-2}{2}}}$ missing two corners in different columns is $\lambda$, then the rank number of $G_{4,n}$ missing two corners in different columns is at least $\lambda+4$.\end{corlem}

\begin{proof}
The Corner Lemma applied to $G_{4,n}$ missing two corners induces one connected component containing a subgrid one of whose missing vertices is a cut set vertex of $G_{4,n}$. The other missing vertex is one of the two missing vertices of $G_{4,n}$. This gives the desired lower bound.
\end{proof}

This corollary illustrated in Figure~\ref{fig:cutcorner} for $n=7$.

Now we are ready to prove our main theorem.  We proceed by casework for the intervals in which $n$ resides.


\begin{case}If $B_0\le n<B_1$, then $r(4,n+1)\ge r(4,n)$ only if $n=B_0+1=2^{k}-2$.\end{case}

\begin{proof}We show an 8-ranking of $G_{4,5}$ with one sticky end in Figure~\ref{fig:4-5sticky} of \ref{app:baseonesticky}.  Then by iteration of the Merging Lemma, there exists a $4k-4$ ranking of $G_{4,2^k-2}$.  However, for $G_{4,2^k-1}$, by the Corner Lemma, any cut set will induce $G_{4,2^{k-1}-1}$ minus a corner.  Because $G_{4,3}$ missing two corners has a rank number of 5, as shown in Figure~\ref{fig:4-3cut1} in \ref{app:corner}, by the corollary of the Corner Lemma, $G_{4,2^{k-1}-1}$ missing two corners has a rank number of at least $4k-7$, so $r(2^k-1)\ge4k-3$.\end{proof}


\begin{case}If $B_1\le n<B_2$, then $r(4,n+1)>r(4,n)$ only if $n=B_1=2^{k}+2^{k-2}-3$.\end{case}

\begin{proof}


\begin{figure}[h]
 \centering
  \subfloat[L-piece]{\label{fig:case2merging1}{\includegraphics[height=1.25in]{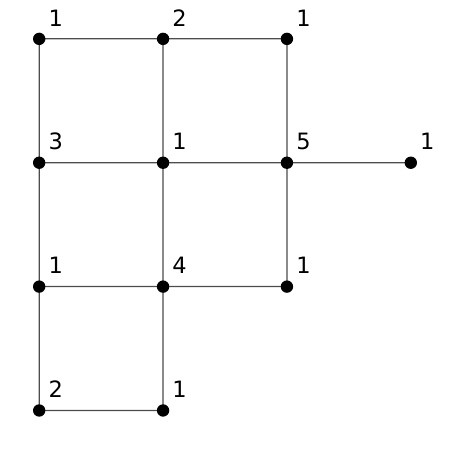}}}
  \subfloat[C-piece]{\label{fig:case2merging2}{\includegraphics[height=1.25in]{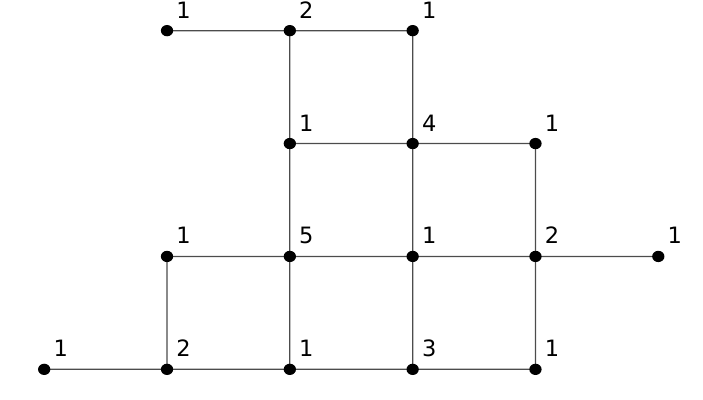}}}
  \subfloat[R-piece]{\label{fig:case2merging3}{\includegraphics[height=1.25in]{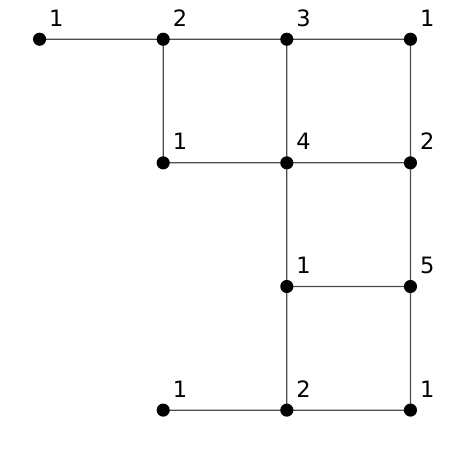}}}
  \caption{Building blocks for constructing $(4k-3)$-ranking of $G_{4,2^k+2^{k-2}-3}$}
  \label{fig:case2merging}
\end{figure}

\begin{figure}[h]
\centering
\includegraphics[height=1.25in]{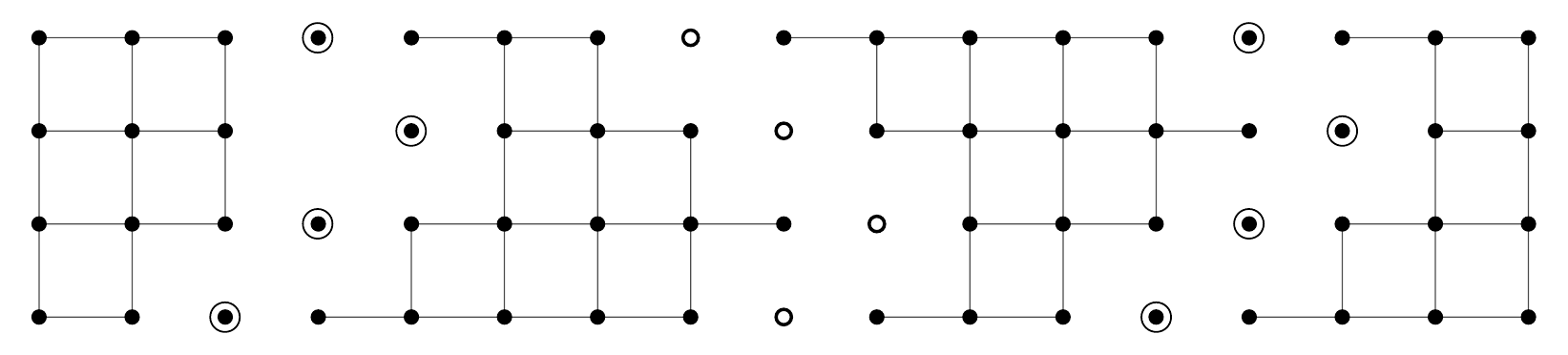}
\caption{Merging Lemma variant for case 2}
\label{fig:4-17merging}
\end{figure}

\begin{figure}[h]
\centering
\includegraphics[height=1in]{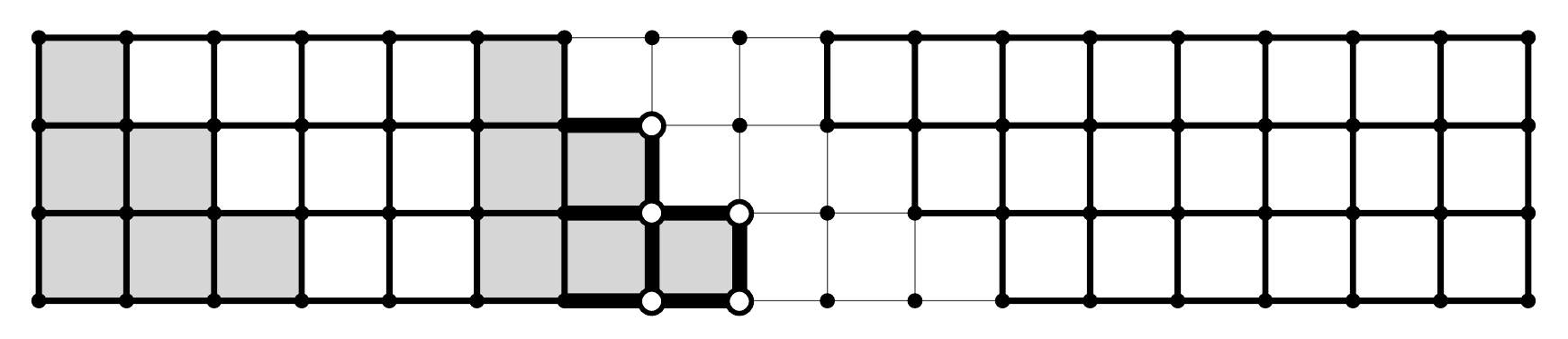}
\caption{Any cut set of $G_{4,2^{k}+2^{k-2}-2}$ must induce a $4\times(2^{k-1}+2^{k-3}-3)$ grid with an escalator end (bolded), any cut set of this must then induce a smaller grid with an escalator end (bolded)}
\label{fig:escalator}
\end{figure}

We use a variant of the Merging Lemma to prove that $r(4,2^k+2^{k-2}-3)=4k-3$. Observe that a $4\times(2^k+2^{k-2}-3)$ grid can be decomposed into one L-piece, $2^{k-2}-2$ C-pieces, and one R-piece, separated by sets of high vertices occupying three columns. An example of this construction for $k=4$ is shown in Figure~\ref{fig:4-17merging}.

Because L-, C-, and R- pieces all have rank numbers of 5, by this construction we know that $r(4,2^k+2^{k-2}-3)=4k-3$. To prove that no $(4k-3)$-ranking exists for $G_{4,2^k+2^{k-2}-2}$, note that any cut set will induce either a $4\times(2^{k-1}+2^{k-3}-3)$ grid with what a set of five extra vertices and eight extra edges (bolded in Figure~\ref{fig:escalator}) which we will call an \emph{escalator end}, or a $4\times(2^{k-1}+2^{k-3}-2)$ grid. Because $r(4,8)=10$, it suffices to show that the former has no $(4k-7)$-ranking. It is evident that $G_{4,2}$ with an escalator end has no 5-ranking, and as Figure~\ref{fig:escalator} illustrates, any cut set of $G_{4,2^k+2^{k-2}-3}$ with an escalator end necessarily induces a $4\times(2^{k-1}+2^{k-3}-3)$ grid with an escalator end, so we are done by the inductive hypothesis. Therefore, $r(4,2^k+2^{k-2}-2)=4k-2$.\end{proof}

\begin{case}If $B_2\le n<B_3$, then $r(4,n+1)>r(4,n)$ only if $n=B_2+1=2^{k}+2^{k-1}-2$.\end{case}

\begin{proof}We show a 6-ranking of $G_{4,3}$ with one sticky end in Figure~\ref{fig:4-3sticky} of \ref{app:baseonesticky}.  Then by iteration of the Merging Lemma, there exists a $(4k-2)$-ranking of $G_{4,2^{k}+2^{k-1}-2}$.


It is not necessarily true that no $(4k-2)$-ranking of $G_{4,2^k+2^{k-1}-1}$ minus two corners in the same row necessarily exists (this is certainly not true of $k=3$), but we will prove that no $(4k-2)$-ranking exists when the corners are diagonally opposite.

\begin{figure}[h]
 \centering
  \subfloat[$A_{2}$]{\label{fig:case3-1}{\includegraphics[height=0.8in]{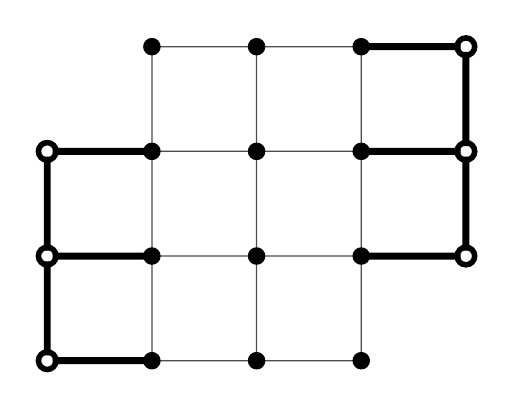}}}
  \subfloat[$B_{2}$]{\label{fig:case3-2}{\includegraphics[height=0.8in]{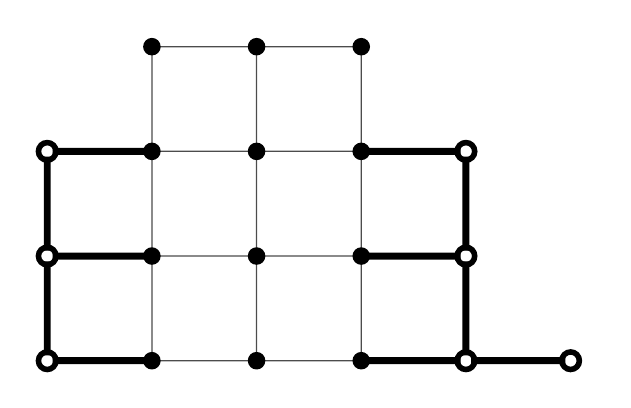}}}
  \subfloat[$C_{2}$]{\label{fig:case3-3}{\includegraphics[height=0.8in]{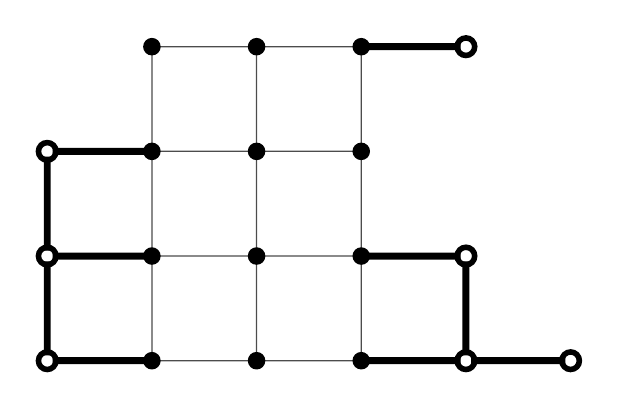}}}
  \subfloat[$D_{2}$]{\label{fig:case3-4}{\includegraphics[height=0.8in]{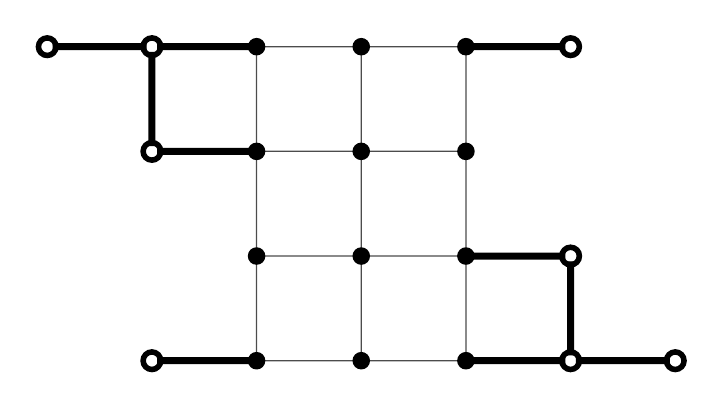}}}
  \subfloat[$E_{2}$]{\label{fig:case3-5}{\includegraphics[height=0.8in]{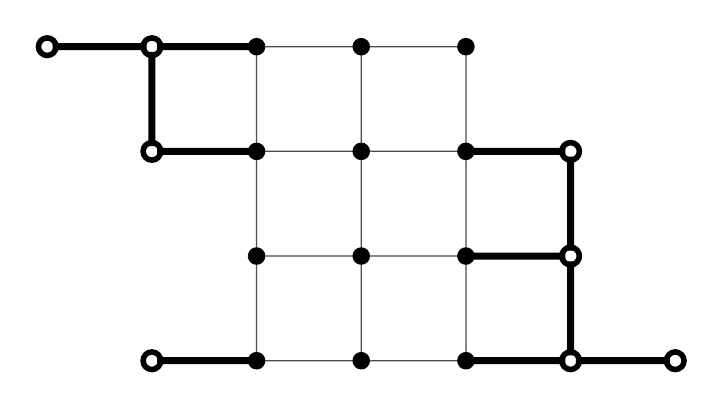}}}  
  \caption{Five sets of ends (bolded) appended to $G_{4,2^2+2^1-3}$ for which there is no 6-ranking}
  \label{fig:case3}
\end{figure}

Figure~\ref{fig:case3} shows a $4\times3$ grid with different sets of extra vertices and edges attached, and for ease of communication, in general we will call a $4\times(2^k+2^{k-1}-3)$ grid with one of these five sets of extra vertices and edges respectively $A_k$, $B_k$, $C_k$, $D_k$, $E_k$. We want to prove that $A_k$, a $4\times(2^k+2^{k-1}-1)$ grid with two diagonally opposite corners, admits no $(4k-2)$-ranking, so to do this, we induct on $k$ to prove a stronger result, that none of the five graphs in Figure~\ref{fig:case3} admits such a ranking. This is certainly true for $k=2$.

\begin{obs}
Any cut set of 

\begin{enumerate}[(i)]
\item $A_k$ or $B_k$ will induce a subgraph containing $A_{k-1}$, $B_{k-1}$, $C_{k-1}$, or $E_{k-1}$.
\item $C_k$ will induce a subgraph containing $A_{k-1}$, $B_{k-1}$, $C_{k-1}$, $D_{k-1}$, or $E_{k-1}$.
\item $D_k$ will induce a subgraph containing $C_{k-1}$, $D_{k-1}$, or $E_{k-1}$.
\item $E_k$ will induce a subgraph containing $B_{k-1}$, $C_{k-1}$, $D_{k-1}$, or $E_{k-1}$.
\end{enumerate}
\end{obs}

By the inductive hypothesis, we are done, so $A_{k}$, which is equivalent to $G_{4,2^{k}+2^{k-1}-1}$ missing two diagonally opposite corners, admits no $(4k-2)$-ranking, and $r(2^k+2^{k-1}-1)=4k-1$.\end{proof}

\begin{case}If $B_3\le n<B_4$, then $r(4,n+1)>r(4,n)$ only if $n=B_3+1=2^k+2^{k-1}+2^{k-2}-2$.\end{case}

\begin{proof}We show a 7-ranking of $G_{4,4}$ with one sticky end in Figure~\ref{fig:4-4sticky} of \ref{app:baseonesticky}.  Then by iteration of the Merging Lemma, there exists a $4k-1$ ranking of $G_{4,2^k+2^{k-1}+2^{k-2}-2}$.

\begin{figure}[h]
 \centering
  \subfloat[]{\includegraphics[width=0.26\textwidth]{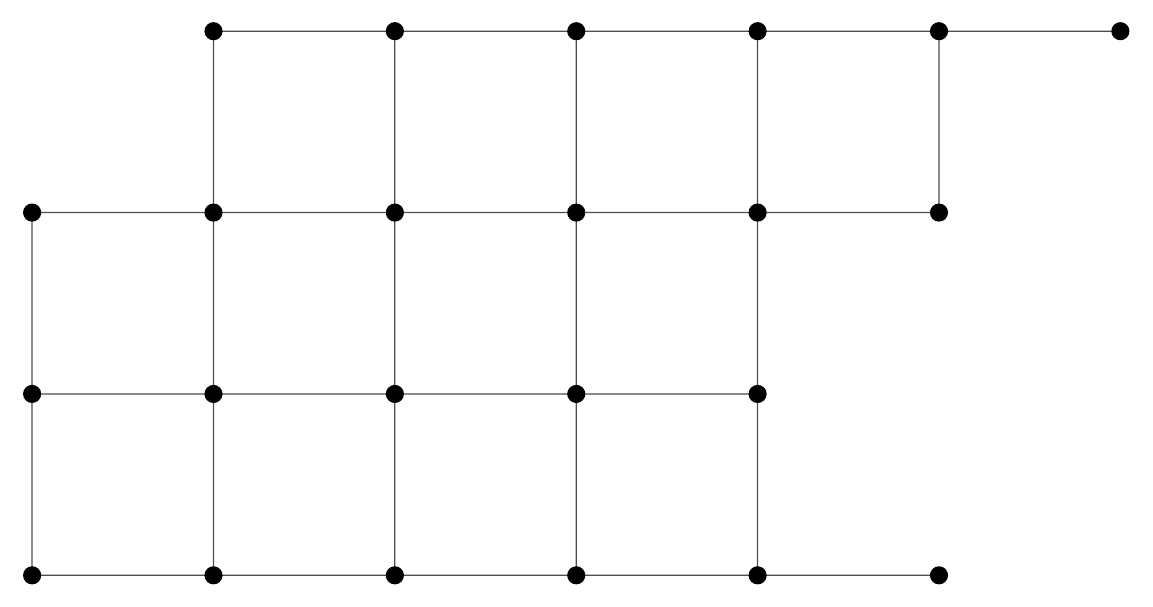}}
  \subfloat[]{\includegraphics[width=0.26\textwidth]{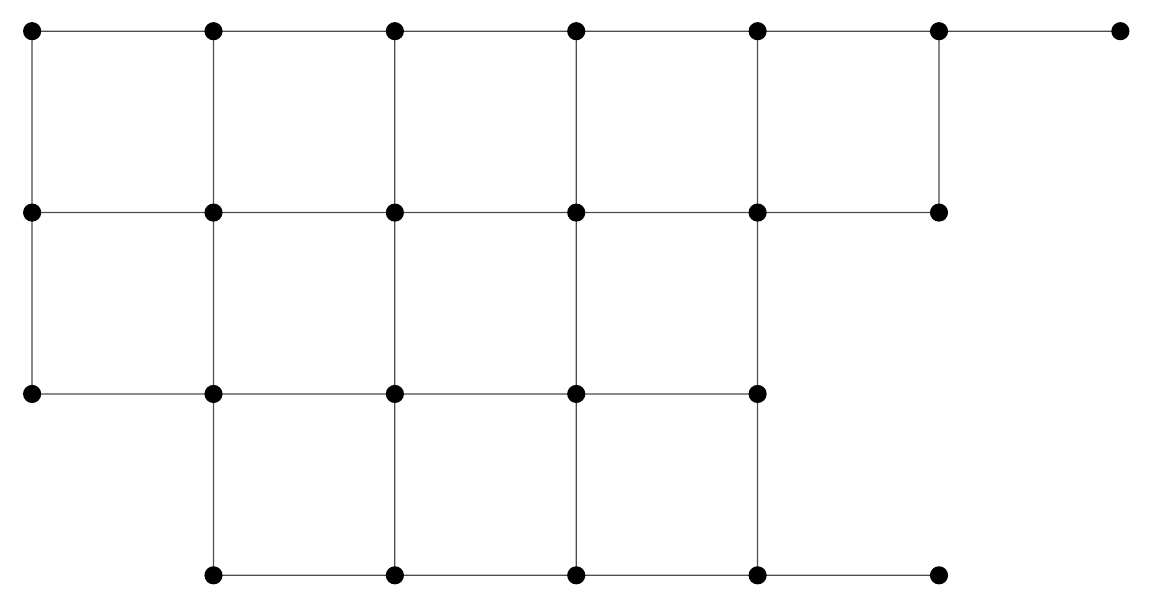}}
  \subfloat[]{\includegraphics[width=0.26\textwidth]{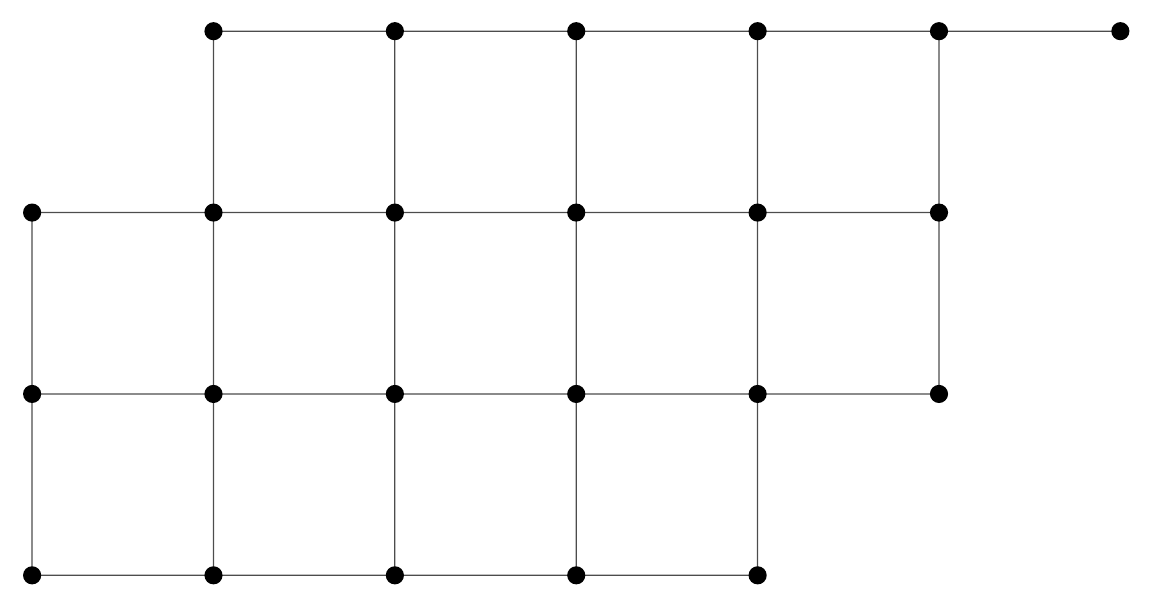}}
  \subfloat[]{\includegraphics[width=0.26\textwidth]{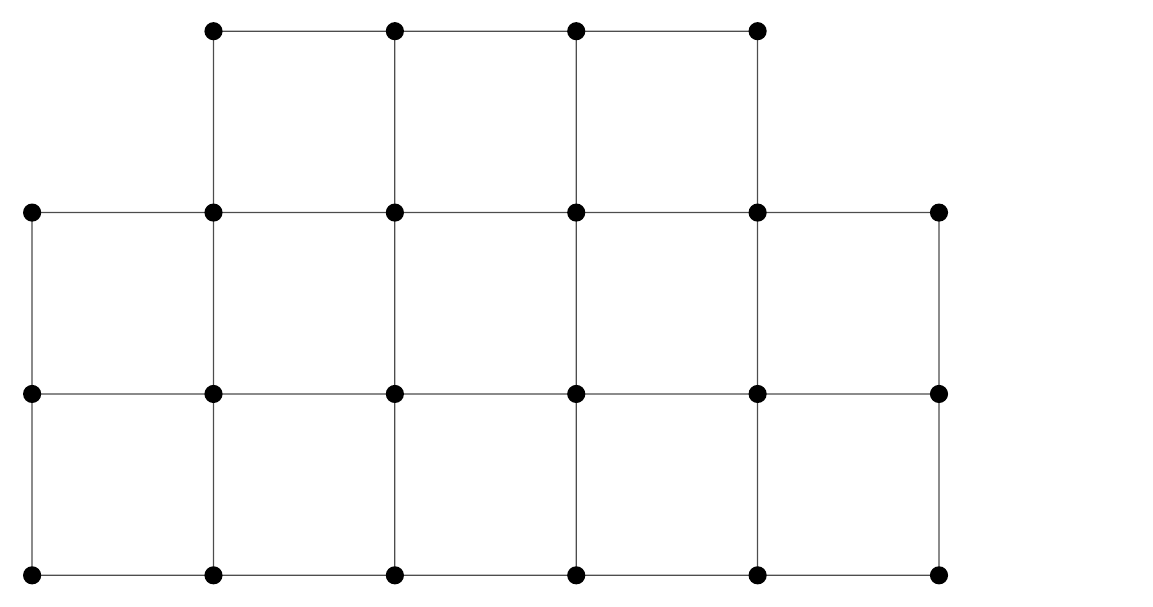}}   
  \caption{Four graphs, one of which any induced subgraph of $G_{4,13}$ must be a subgraph}
  \label{fig:13must}
\end{figure}

However, for $G_{4,2^k+2^{k-1}+2^{k-2}-1}$, by the Corner Lemma, any cut set will induce a subgraph $4\times(2^{k-1}+2^{k-2}+2^{k-3}-1)$ grid minus a corner. By  the bounds of\eqref{eq:bounds}, $r(4,13)\le 12$, so the rank number of $G_{4,13}$ minus two corners in different columns is at most 12. Observe that any cut set will induce a subgraph contained in one of the graphs in Figure~\ref{fig:13must}, and we prove in \ref{app:case4} that none of these has a 7-ranking, so $r(4,13)=12$, as desired.

Therefore, $r(4,13)=12$, as desired. By the corollary of the Corner Lemma, $G_{4,2^{k-1}+2^{k-2}+2^{k-3}-1}$ missing two corners has a rank number of at least $4k-4$, so $r(2^k+2^{k-1}+2^{k-2}-1)\ge4k$.\end{proof}

We have completely determined the rank number for $4\times n$ grid graphs.

\section{Upper Bound for General Grids}

In \cite{alpert}, Alpert showed that $r(m,n)\le m+r(m,\ceil{\frac{n-1}{2}})$ by constructing a ranking with $m$ high vertices in the middle column and $r(m,\ceil{\frac{n-1}{2}})$ smaller colors ranking the larger connected component \cite{alpert}.  Here we present an improved upper bound.

For simplicity, let $\chi_r(\tria{n})=\tri{n}$.  The subgraphs induced by removing a diagonal cut set of $m$ vertices from the middle of $G_{m,n}$ are $G_{m,\floor{\frac{n-m}{2}}}$ with one sticky end and $G_{m,\ceil{\frac{n-m}{2}}}$ with one sticky end, respectively, as shown in Figure~\ref{fig:Upper}.  The rank number of the latter is greater than that of the former.  Furthermore, $G_{m,\ceil{\frac{n-m}{2}}}$ with one sticky end can be broken up into $G_{m,\ceil{\frac{n-m}{2}}-1}$ and $\tria{m}$, so that $G_{m,\ceil{\frac{n-m}{2}}-1}$ is colored with the integers from 1 to $r\left(m,\ceil{\frac{n-m}{2}}-1\right)$, inclusive, while $\tria{m}$ is colored with the integers from $r\left(m,\ceil{\frac{n-m}{2}}-1\right)+1$ to $r\left(m,\ceil{\frac{n-m}{2}}-1\right)+\tri{m}$, inclusive.

\begin{figure}[h]
 \centering
  \subfloat[Vertical cut set (old)]{\includegraphics[width=0.4\textwidth]{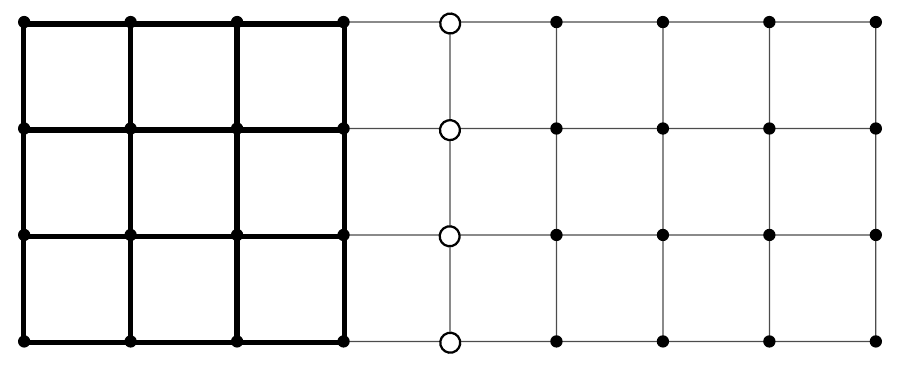}}
\hspace{1cm}
  \subfloat[Diagonal cut set (new)]{\includegraphics[width=0.4\textwidth]{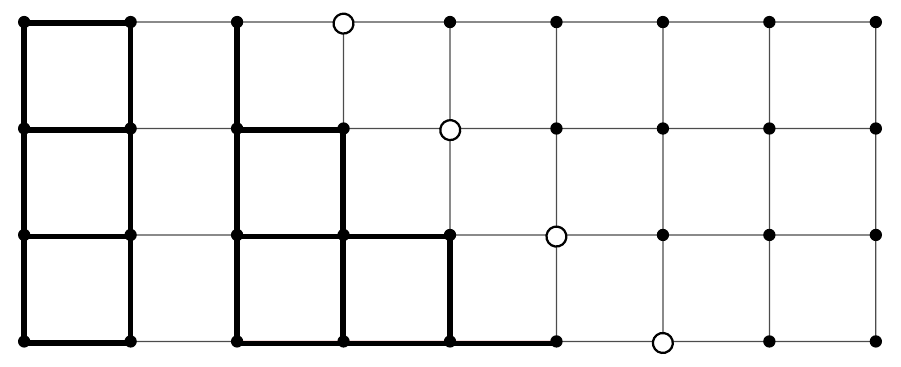}}
  \caption{Comparison of cutting methods}
  \label{fig:Upper}
\end{figure}

In \cite{alpert}, Alpert showed that $\tri{n}\le n-1+\tri{\ceil{ \frac{n-2}{2} }}.$  For $n\ge 3$, we can iterate this to get $\tri{n}\le 2n-2\floor{\log_2(n+1)}+1$ (see \ref{app:tri}).  Using this, we get a new upper bound of $$r(m,n)\le 3m-2\log_2(m+1)+1+r\left(m,\ceil{\frac{n-m}{2}}-1\right).$$  For $n\le N$, where $N=O\left(m^{\frac{3}{2}}\right)$, our bound is tighter (see \ref{app:gen}).

\section{Lower Bound for Square Grid Graphs}


\begin{thm}
For $m\ge 5$, $r(m,m)\ge m+r\left(\ceil{\frac{2m}{5}}-1,\ceil{\frac{2m}{5}}-1\right).$
\end{thm}

First note that for $|C(f)|<m$, one of the subgraphs induced by removing the cut set will contain $\sq{\ceil{\frac{m}{2}}}$. Without loss of generality, let the top and bottom rows of the square grid both contain at least one cut set vertex (in rough terms, the cut set ``goes from top to bottom''). Furthermore, let the number of consecutive columns that the cut set occupies in $S_m$ be $k$. Clearly our lower bound holds for $k=1$, so now, we only consider cut sets at least two sets wide.  

We first define some terms, examples of which are illustrated in Figure~\ref{fig:cutpath}. Define two cut set vertices to be \textit{cut-neighbors} if their column and row indices each differ by at most 1.  A \textit{cutpath} between two points $v_0$ and $v_n$ is the sequence of cut set vertices $[v_0,v_1,...,v_{n}]$, where $v_i$ and $v_{i+1}$ are cut-neighbors for all $0\le i<n$ and a $\textit{sub-cutpath}$ to be any subsequence of the cutpath. Because some rows will contain multiple cut set vertices, there is no bijection between a cut set and its column sequence.  Instead, call the bottommost row of the cut set row 0; we will denote a vertex in row $r$ and column $c$ as vertex $(r,c)$ and a sub-cutpath as $[(r_0,c_0),(r_1,c_1),...,(r_n,c_n)]$.  

 Note that for $|C(f)|=m$, there is a bijection between the sequence of vertices in a cutpath and the sequence of column indices of those respective vertices.  Denote this \textit{column sequence} by $[c_0,c_1,...,c_{n}]$, where $c_i$ is the index of the column that $v_i$ occupies.

We will let the index of the leftmost column be 0 and that of the rightmost column be $k-1$.  Then define a \textit{left-diagonal cutpath} (\textit{right-diagonal cutpath}) to be a cutpath with column sequence $[0,1,,...,k-1,k-2,...,1,0]$ ($[k-1,k-2,...,0,1,...,k-2,k-1]$).  

Define a \textit{left-deviation} (\textit{right-deviation}) to be a subsequence of the form $[c_i,c_{i+1}]$ in the cutpath's column sequence, where $v_{i+1}$ is in a row higher (lower) than that of $v_i$ and $c_{i+1}-c_i\le 0$ ($c_{i+1}-c_i\ge 0$) and the closest vertices preceding and following it in the cutpath that lie on columns 0 and $k-1$ lie, respectively, on column 0 (column $k-1$) and column $k-1$ (column 0) .

Define a \textit{deformation of $v$} as the shifting of $v$ and all vertices following $v$ in the cutpath up one unit, as long as $v$ is still a cut-neighbor of the vertex preceding it in the cutpath.  

For a rectangular subgrid $G_{x,y}$ with its leftmost column on the leftmost column of the original $\sq{m}$ grid, define its \textit{first extension} to be its supergraph $G_{x,y+1}$ minus the rows containing vertices on the cutpath.  The $i$th extension is defined to be the extension of the $(i-1)$th extension.  Note that if a given subgrid in $S$ must exist, then its extension must also exist.  An example is shown for $\sq{10}$ and $k=7$ in Figure~\ref{fig:cutpath}.  Finally, define the \emph{$k$th subgrid extension} to be the subgrid $G_{x,y}$ in $S_{m,k}$, defined in Lemma \ref{lem:smk}, so that $(x,y)=2k-3-2i,\ceil{\frac{m-k}{2}}+i$.

To prove our lower bound for $r(m,m)$, we will essentially show that for any given $k$, the cut set that gives a ranking of fewest labels is a diagonal cutpath (obviously, orientation is irrelevant). We will first deal with cut sets of $m$ vertices, for which we will show in Lemma 6 that for a given $k$, a certain family of subgrids is guaranteed to exist for any cut set of exactly $m$ vertices (it will eventually be from this family that we derive the recursive term in our lower bound). In particular, we will show that cut sets with deviations can be seen as equivalent to diagonal cutpaths. We will then deal with cut sets of more vertices with the Large Cut Set Lemma, showing through Lemmas 7 and 8 that a ranking corresponding to a cut set of more than $m$ vertices will use no fewer colors than a particular ranking with a cut set of exactly $m$ vertices. The recursive term in our lower bound then comes from computation of the largest square subgrid that exists in the families of induced subgrids from Lemma 6 ranging over all values of $k$.

\begin{lem}
For $k\le\floor{\frac{m+2}{3}}$, any cut set of $m$ vertices will induce subgraphs $G_{x,y}$, where \begin{equation*}(x,y)\in\left\{\left(m,\ceil{\frac{m-k}{2}}\right)\right\}\cup\left\{\left(x',y'\right): x',y'>0, \ x'=2k-3-2t, \ y'=\ceil{\frac{m-k}{2}}+t\right\}.\end{equation*}\label{lem:smk}
\end{lem}

Let $S_{m,k}$ be the set of subgraphs $G_{x,y}$ induced by removing the cut set. 

\begin{proof}Because $|C(f)|=m$, each row in this path will contain exactly one vertex.  

Pick two vertices $v$ and $v'$ of the cut set that occupy the leftmost column so that the cutpath of cut set vertices connecting them contains at least one vertex in the rightmost column and no other vertices in the leftmost column.  The first condition has to be true or the cut set occupies fewer than $k$ columns.  The column sequence will be of the form $[0,c_1,c_2,...,c_i,c_{i+1}=k-1,c_{i+2},...,c_j=0]$.

\begin{figure}[h]
\begin{center}
 \subfloat[A pair of cut-neighbors is circled; the cutpath is connected by a dotted line. The left-deviation is bolded.  The first extension of the bolded $5\times 4$ subgrid is the dotted $3\times 5$ subgrid.  The second extension is the white $1\times6$ subgrid.  Left-diagonal cutpath vertices are white.]{\label{fig:cutpath}\includegraphics[width=0.3\textwidth]{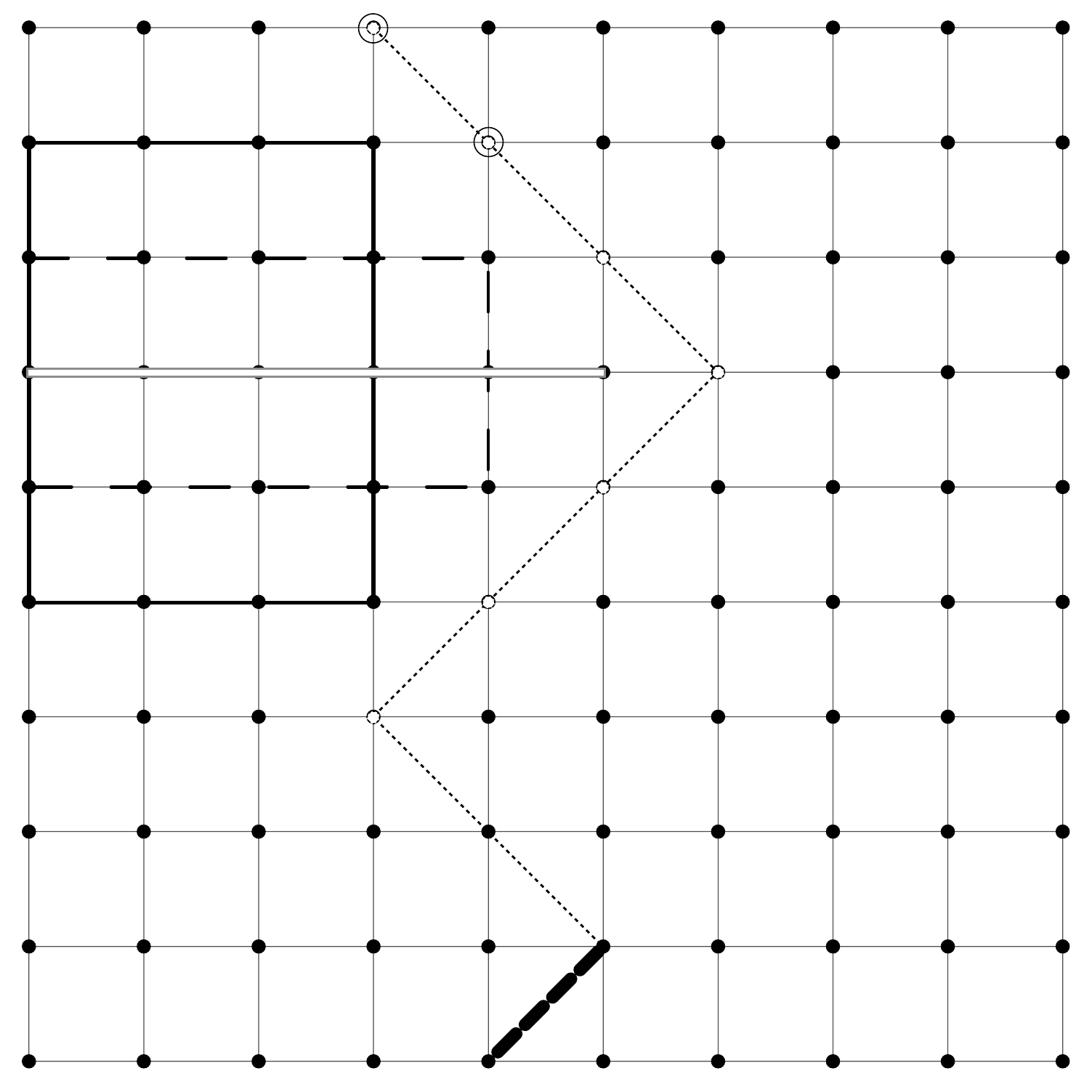}}
\hspace{0.5cm} 
 \subfloat[A cutpath with no deviations (solid line) induces a connected component to the left that is the subgraph (grey vertices) of the connected component induced by a cutpath with deviations (dotted line).  Extra vertices are bolded; deviations are bolded.]{\label{fig:deviation}\includegraphics[width=0.3\textwidth]{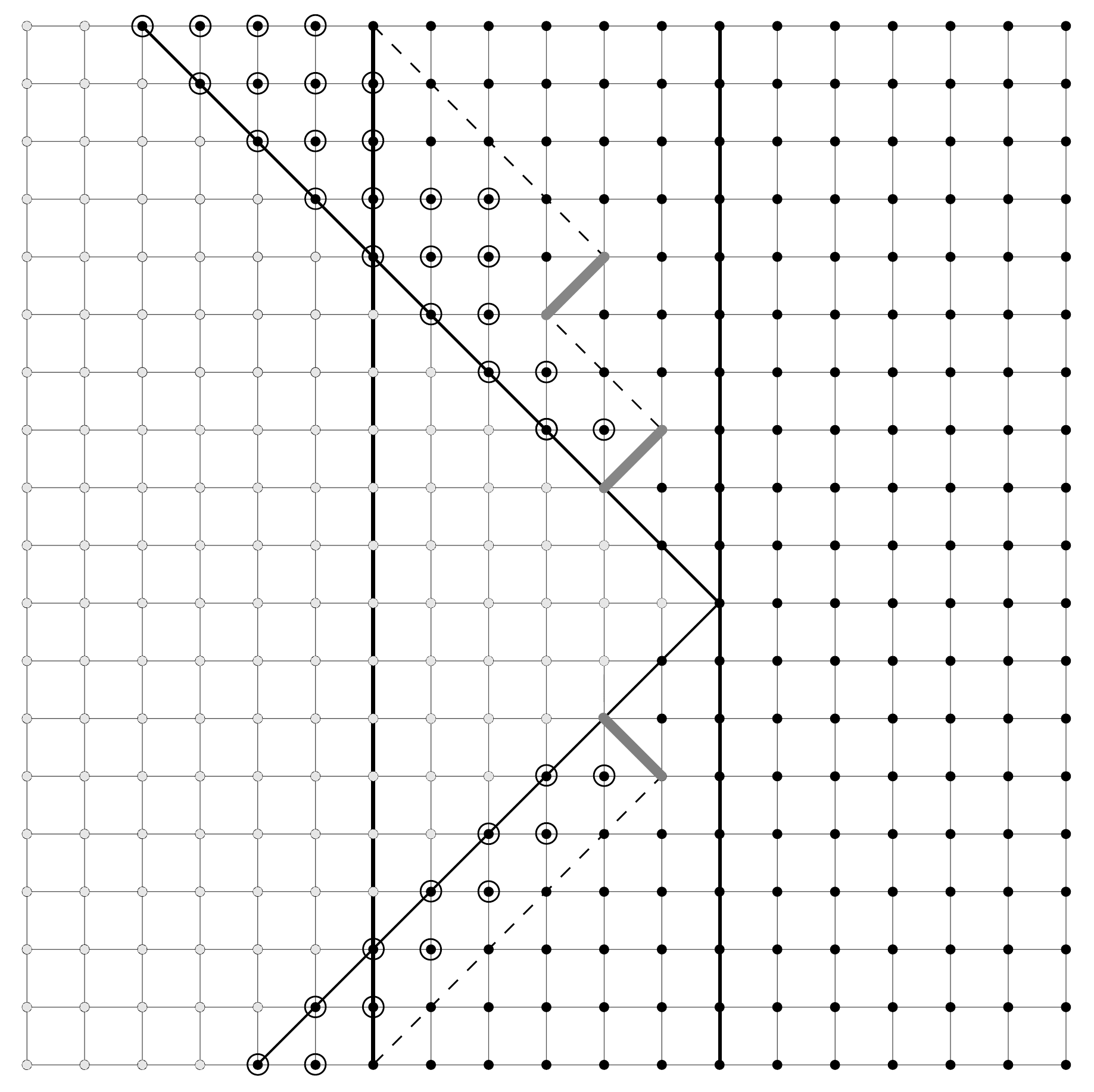}}
\hspace{0.5cm}
 \subfloat[Bolded vertices represent components that the long diagonal splits the columns into.  One of these must contain $\floor{\frac{m-k}{2}}$ consecutive vertices.  This yields a $G_{m-1-\floor{\frac{m-k}{2}},\floor{\frac{m-k}{2}}+1}$ subgrid.  There are exactly $k-\left(\ceil{\frac{m-k}{2}}+1\right)-1$ (circled) small extensions (dotted subgrid).]{\label{fig:smallextension}\includegraphics[width=0.3\textwidth]{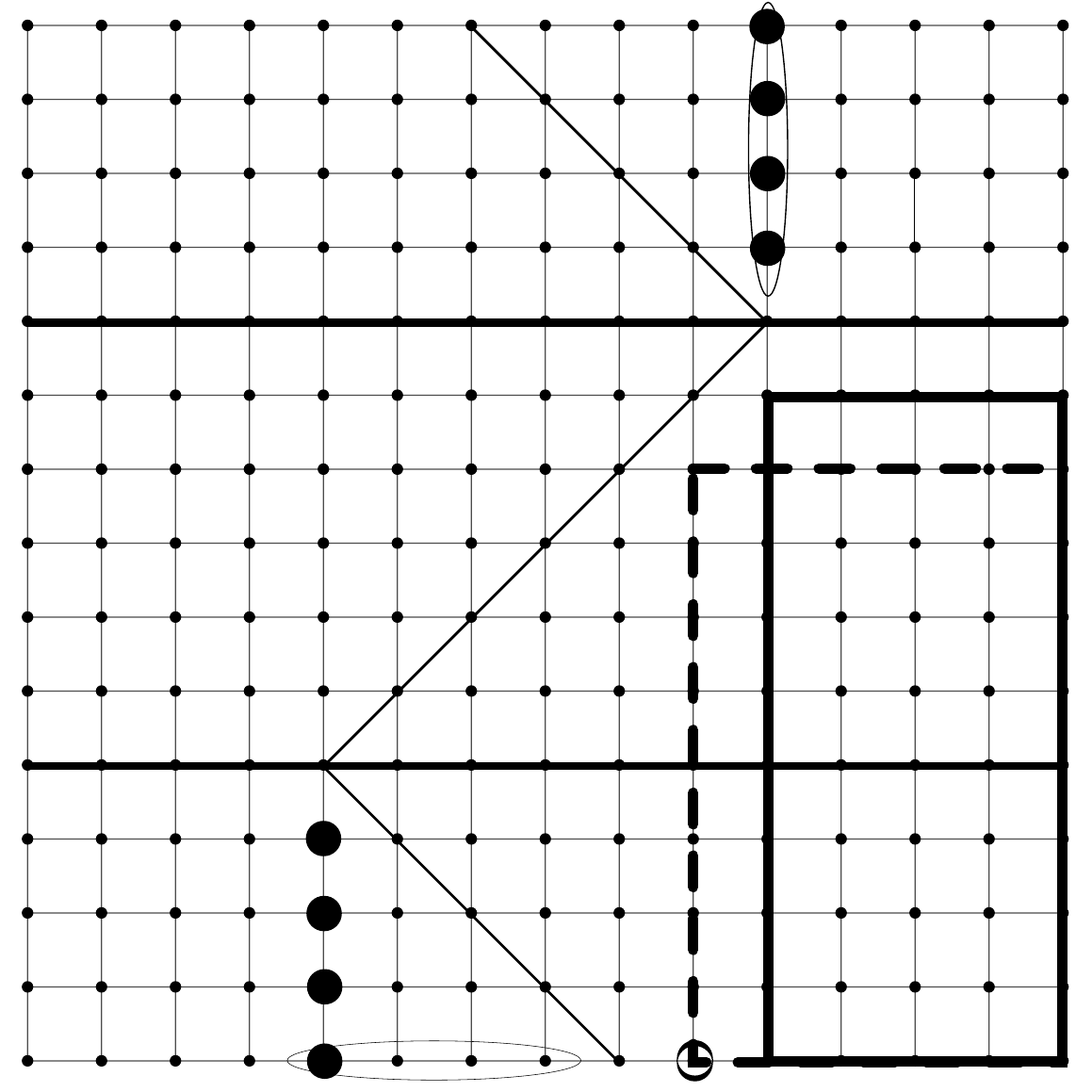}}
\end{center}
\caption{Examples of deviations and extensions and small extensions.}
\label{fig:deviation1}
\end{figure}

Consider the bottommost left-deviation $[c_l,c_{l+1}]$ and the topmost right-deviation $[c_r,c_{r+1}]$.  Note that extending the lines of slope $-1$ and 1 containing $c_{l+1}$ and $c_{r+1}$, respectively gives a diagonal cutpath falling within the boundaries of the cutpath, so the left connected component induced by removing a diagonal cutpath will be the subgraph of any cutpath with at least one deviation.  An example of this argument is shown in Figure~\ref{fig:deviation1}.  Thus, for $|C(f)|=m$, we will only consider diagonal cutpaths.

Clearly, removing the cut set's $k$ columns induces a $G_{\ceil{\frac{m-k}{2}},m}$ subgraph.  Consider the part of the induced subgraph in the same rows as the vertices in the left-diagonal cutpath.  Then a rectangle with at least $\ceil{\frac{m-k}{2}}$ columns exists between the diagonal cutpath's two endpoints, and it occupies all rows between those of the cutpath's endpoints, the rectangle has a width of $2k-3$ rows.  Note that because a diagonal cutpath will occupy two vertices in each column of the grid except the last column, an extension of a grid $G_{x,y}$ that must exist will have dimensions $(x-2)\times(y+1)$, giving us the desired set of rectangular subgrids that are guaranteed to exist for $G_{m,m}$.\end{proof}

\begin{lem3}For each ordered pair $(x,y)$ so that there is an $x\times y$ subgrid in $S_{m,k}$ for a cut set of $m$ vertices, if $|C(f)|>m$, then there is at least one subgrid $G_{x,y}$ of $\sq{m}$ that contains at most $|C(f)|-m$ vertices.\end{lem3}

To prove the Large Cut Set Lemma, we first prove two lemmas.

\begin{lem}
If the greatest lower bound on the rank number of all subgraphs induced by cut sets with exactly $m$ vertices is $r_{\text{min}}$, then the subgraph induced by  cut set where at least two cut-neighbors are in the same column, has a rank number of at least $r_{min}$.
\end{lem}

\begin{proof}
As shown in Figure~\ref{fig:deformations}, there are nine configurations for a set of three consecutive columns where the middle column contains two cut-neighbors.  

\begin{figure}[h]
\begin{center}
 \subfloat[]{\label{fig:vertical1}\includegraphics[height=0.18\textheight]{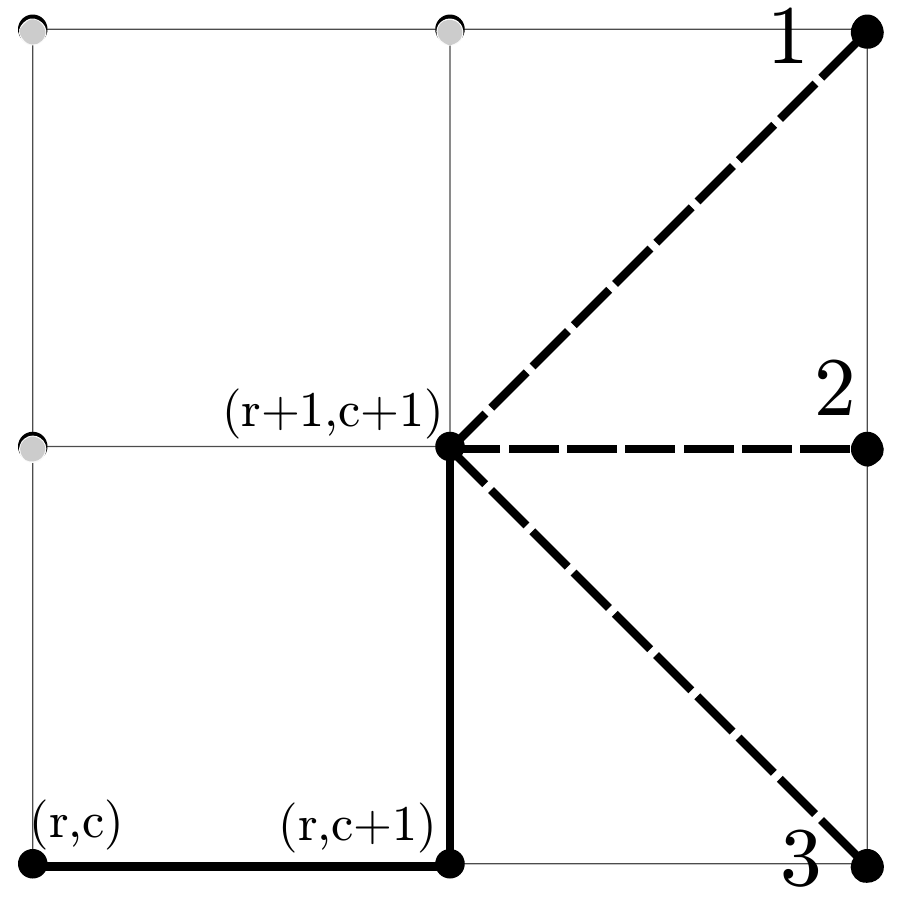}}
\hspace{0.8cm}
 \subfloat[]{\label{fig:vertical2}\includegraphics[height=0.18\textheight]{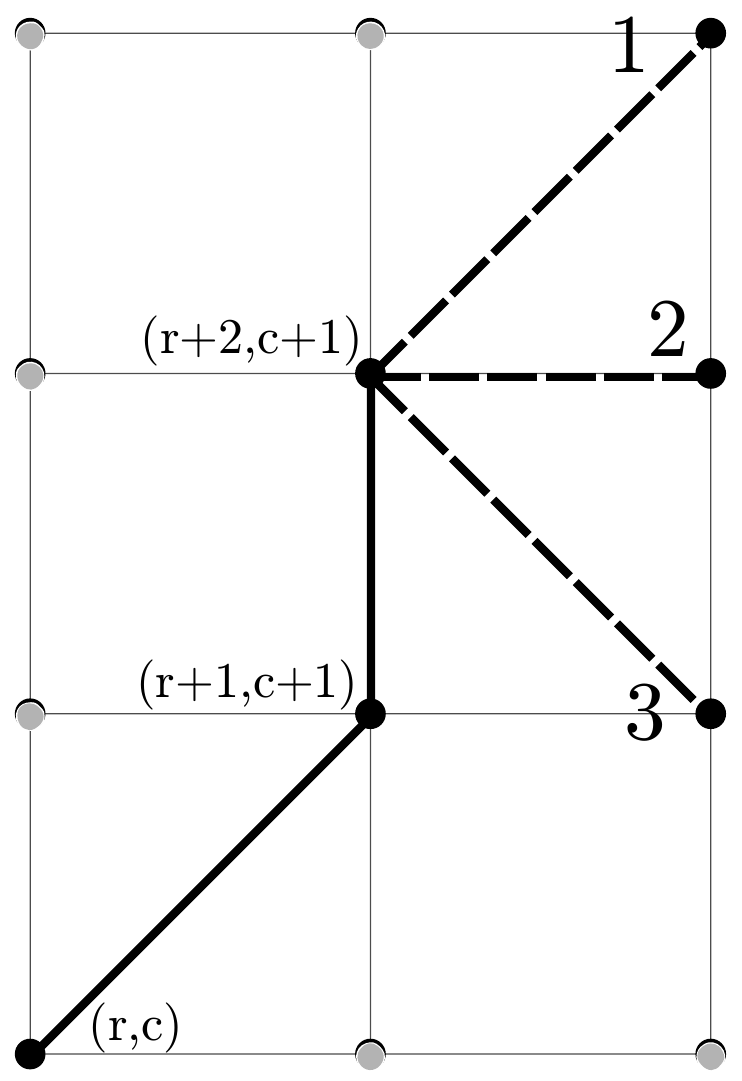}}
\hspace{0.8cm}
 \subfloat[]{\label{fig:vertical3}\includegraphics[height=0.18\textheight]{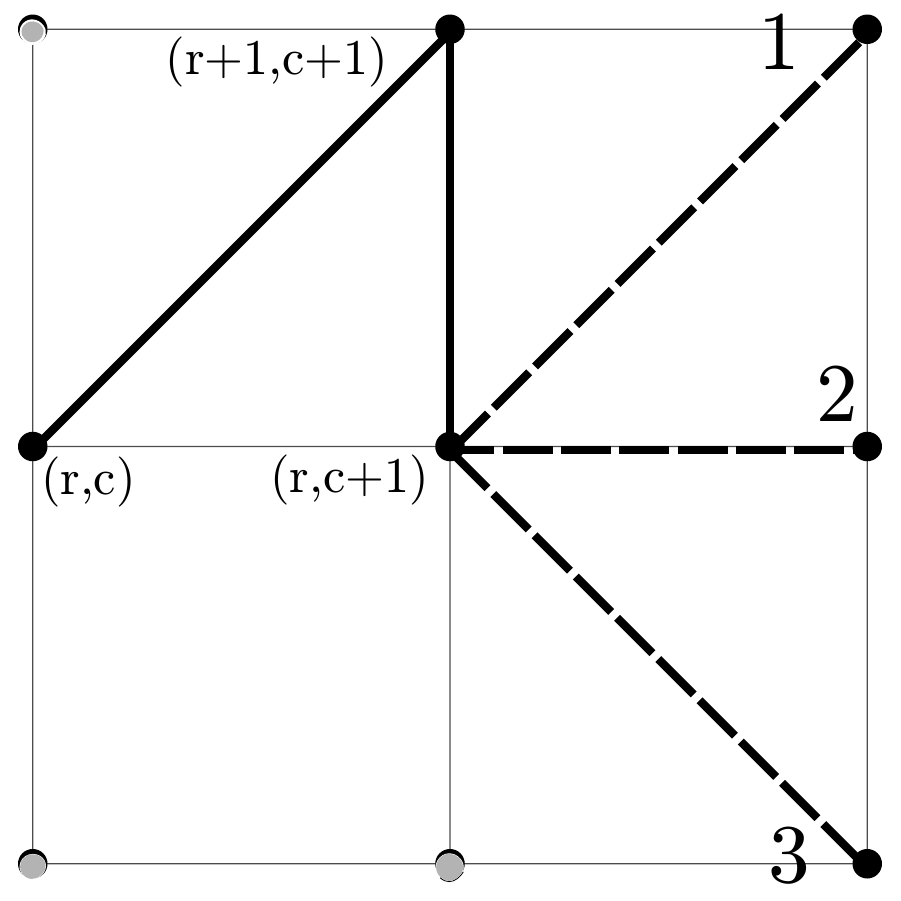}}
\end{center}
\caption{Cutpaths not obtainable by a set of deformations still satisfy the lower bound.}
\label{fig:deformations}
\end{figure}

In the first three cases shown in Figure~\ref{fig:vertical1}, the sub-cutpath $[(r,c),(r,c+1),(r+1,c+1)]$ is equivalent to a sub-cutpath of $[(r,c),(r+1,c+1)]$ with an extra cut vertex $(r,c+1)$.  If the induced subgraph of the cut set containing $[(r,c),(r+1,c+1)]$ has rank number of at least $r_{\text{min}}$, adding $(r,c+1)$ to the cut set will increase $|C(f)|$ by one while decreasing the rank number of the induced subgraph to the right of the cut set by at most 1, and it will still satisfy the lower bound. 

Case 1 of Figure~\ref{fig:vertical2} is equivalent to the statement that cutpaths with deviations still satisfy the lower bound, proven in Lemma 6.  Case 2 of Figure~\ref{fig:vertical2} follows similar to the three cases in Figure~\ref{fig:vertical1}.  In Case 3 of Figure~\ref{fig:vertical2}, sub-cutpath $[(r,c),(r+1,c+1),(r+2,c+1),(r+1,c+2)]$ is equivalent to $[(r,c),(r+1,c+1),(r+1,c+2)]$ with an extra vertex $(r+2,c+1)$, and if the induced subgraph of the cut set containing $[(r,c),(r+1,c+1),(r+1,c+2)]$ has a certain rank number, then adding $(r+2,c+1)$ will increase $|C(f)|$ by one while decreasing the rank number of the subgraph to the left of the cut set by at most one.

The three cases in Figure~\ref{fig:vertical3} follow by reasoning similar to that for case 3 of Figure~\ref{fig:vertical2}.
\end{proof}

Thus, we only consider when cut-neighbors are in different columns. Observe that such cutpaths can be obtained through some sequence of deformations of and additions of vertices to a cutpath of length $m$.  In particular, $d$ deformations of vertices on a left-diagonal cutpath with endpoints $(0,0)$ and $(0,2k-2)$ will move the original endpoint to $(0,2k-2+d)$.  Then there are at least $|C(f)|+d$ vertices in the new cut set $C'(f)$.  In general, if cut set $C'(f)$ contains a sub-cutpath that is obtained through a set of $d$ deformations of $C(f)$, $C'(f)\ge C(f)+d$.

\begin{lem}
If the subgraph induced by a cut set, where all cut set vertices after $(0,2k-1+d)$ in the cutpath are part of a diagonal cutpath, has a rank number of at least $r(\ceil{\frac{2m}{5}}-1,\ceil{\frac{2m}{5}}-1)$, then the subgraph induced by a cut set, where the vertices after $(0,2k-1+d)$ are obtained by some deformations of a diagonal cutpath, has a rank number of at least $r(\ceil{\frac{2m}{5}}-1,\ceil{\frac{2m}{5}}-1)$.
\end{lem}

\begin{proof}
Consider a sub-cutpath beginning with $(0,2k-1+d)$ and ending with the first vertex in the cutpath that lies on the line of slope 1 of vertices containing $(k,k)$, and let the subgraph induced by a cut set $C(f)$ that has this sub-cutpath contain $\sq{\ceil{\frac{2m}{5}}-1}$. Let $C'(f)$ be the cut set with the same vertices as $C(f)$ before $(0,2k-1+d)$ but with an additional deformation on the sub-cutpath after $(0,2k-1+d)$. Then the extra deformation in $C'(f)$ will remove at most one more vertex from each subgrid in $S_{m,k}$ than will $C(f)$.  However, $C'(f)\ge C(f)+1$, so the rank number of the induced subgraph will stay the same or increase.
\end{proof}

We now prove the Large Cut Set Lemma.

\begin{proof}
We prove by induction that the cut set of $|C(f)|>m$ vertices will occupy at most $|C(f)|-m=d$ vertices in at least one subgrid in $S_{m,k}$.  Note that the extension of a subgrid will contain at most one more cut vertex than the subgrid because we are considering cut sets where columns contain no cut-neighbors in the same column.  In columns after the $\floor{\frac{d-1}{2}}$th column, the diagonal cutpath beginning with $(0,2k-1+d)$ occupies rows lower than a left-diagonal cutpath with endpoints $(0,0)$ and $(0,2k-1)$ would.  Before this point, the $j$th extension contains at most two more cut set vertices than the $(j-1)$th extension, so the $\floor{\frac{d-1}{2}}$th extension will have at most $1+2\floor{\frac{d-1}{2}}\le d$ cut set vertices, and because the sub-cutpath beginning with $(0,2k-1+d)$ is diagonal, it will have a vertex in the bottom row.  Assume $G_{\ceil{\frac{m-k}{2}},2m-3-2d}$ contains at most $d$ cut set vertices and a cut set vertex $v$ in the bottom row.  The extension of this will not contain $v$, but, again, it will contain at most one more cut set vertex than the subgrid, completing the induction.
\end{proof}  

So to prove our main result, that $r(m,m)\ge m+r(\ceil{\frac{2m}{5}}-1,\ceil{\frac{2m}{5}}-1)$, it suffices to prove this for $|C(f)|=m$.

\begin{proof}
We first show for each cut set column width, the subgrids in $S_{m,k}$ contain $G_{\ceil{\frac{2m}{5}}-1,\ceil{\frac{2m}{5}}-1}$.  In \ref{app:square}, we show that for $k=\floor{\frac{m-1}{5}}+2$, the 0th subgrid extension must contain $r(\ceil{\frac{2m}{5}}-1,\ceil{\frac{2m}{5}}-1)$ as a subgraph, so for $k=\floor{\frac{m-1}{5}}+2$, the 0th subgrid extension must contain $\sq{\ceil{\frac{2m}{5}}-1}$ as a subgraph, so for $k\le\floor{\frac{m-1}{5}}+2$, the 0th subgrid extension must contain this subgraph.  

We now prove by induction that for $\floor{\frac{m-1}{5}}+2<k\le\ceil{\frac{m+2}{3}}$, one extension for each $k$ must contain $\sq{\ceil{\frac{2m}{5}}-1}$.  If for $i$ columns, the $n$th extension contains $\sq{\ceil{\frac{2m}{5}}-1}$, then the $n$th extension for $i+1$ columns will have at least $\ceil{\frac{2m}{5}}-1$ columns but exactly $\ceil{\frac{2m}{5}}+1$ rows.  The $(n+1)$th extension will have $\ceil{\frac{2m}{5}}$ rows and $\ceil{\frac{2m}{5}}$ columns.  

However, note that for $k>\floor{\frac{n+2}{3}}$, a diagonal cut set will not necessarily exist.  If $|C(f)|=m$, then there will be a sub-cutpath with column sequence $[k-1,k-2,..,1,0]$, because a deviation will only increase the dimensions of the subgrids that are guaranteed to exist.  Then either the left- or rightmost column of the cut set will contain at most $\floor{\frac{m-k}{2}}$ consecutive vertices not belonging to the cut set.  Of the remaining vertices on that column, only one is in the cut set, so the column will contain at least $m-1-\floor{\frac{m-k}{2}}$ consecutive vertices not belonging to the cut set.  This guarantees the existence of a $G_{m-1-\floor{\frac{m-k}{2}},\floor{\frac{m-k}{2}}+1}$ subgrid.  Note that for $i< k-(\ceil{\frac{m-k}{2}}+1)$, the $(i-1)$th extension of a subgrid will have exactly one more row than the $i$th extension.  Call such an extension a \emph{small extension}.  For $i\ge k-(\ceil{\frac{m-k}{2}+1})$, the $(i-1)$th extension will have two more rows than the $i$th extension.

If $m-k$ is even, then after $k-(\ceil{\frac{m-k}{2}}+1)-1$ small extensions, we have a subgrid $G_{m-k+1,k-1}$.  We prove by induction that if for $i$ columns, the $n$th subgrid extension contains $\sq{\ceil{\frac{2m}{5}}-1}$, where $i\le k-(\ceil{\frac{m-k}{2}}+1)-1$, then for $i+1$ columns, the $(n+1)$th subgrid extension contains $\sq{\ceil{\frac{2m}{5}}-1}$.  Indeed, this is true for $k=\floor{\frac{m+2}{3}}+1$ columns.  If this is true for $i$ columns, then for $i+1$ columns, the number of small extensions is at least one more than the number for $i$ columns, so the $(n+1)$th extension is small and yields $G_{m-k+1,k-1}$, which contains $\sq{\ceil{\frac{2m}{5}}-1}$ by the inductive hypothesis.

If $m-k$ is odd, then after $k-(\ceil{\frac{m-k}{2}}+1)-1$ extensions, we have a subgrid $G_{m-k+2,k-2}$.  We can prove this contains $\sq{\ceil{\frac{2m}{5}}-1}$ similarly. \end{proof}

\begin{cor}$r(m,m)\ge\frac{5}{3}m-\frac{25}{9}$.\end{cor}

\begin{cor}$\emph{tri}_n\ge\frac{5}{3}\floor{\frac{n}{2}}-\frac{34}{9}$.\end{cor}

\begin{proof}
This follows from the fact that $\text{tri}_n$ contains $G_{\floor{\frac{n}{2}}}$.
\end{proof}

The rank number for triangle and square grids is now bounded between two linear functions, setting the groundwork for completely determining $r(m,m)$ and $\tri{n}$.

\section{Conclusion}

In this paper, we have completely determined the rank number for $4\times n$ grids, improved the upper bounds for the rank numbers of general grids, and improved the lower bounds for the rank numbers of square grids from a logarithmic to linear.  As a corollary, we have improved the lower bounds for the rank numbers of triangle grids from logarithmic to linear as well.

Using our lower bound for square grid graphs, we may be able to determine the rank number function for general grid graphs.  Firstly, we can generalize the Merging Lemma to all grid graphs: if $G_{m,n}$ with a sticky end has a $\lambda$-ranking and $G_{m,n-1}$ has a $\lambda$-ranking, then $G_{m,m+2n-1}$ has a $(\lambda+m)$-ranking and $G_{m,3m+4n-2}$ has a $(\lambda+2m)$-ranking.  Furthermore, the motivation behind the Corner Lemma may be extended to grids of higher dimensions.  Finally, we conjecture that removing two corners in different columns from a sufficiently long grid will preserve the rank number of the grid.  One avenue of further research is thus to prove or disprove this conjecture and to use this in conjunction with our lower bound and the generalization of the Merging Lemma in order to completely determine $r(m,n)$.

For sufficiently long grid graphs $G_{m,n}$, where $m\le n$, we also conjecture that there always exist $r(m,n)$-rankings where $|C(f)|=m$.  Moreover, the new ideas about the cut set that we discuss in our proof for square grids are interesting objects of further study, in particular whether deformations or deviations will not yield rankings of fewer labels for other graphs.


\section{Acknowledgments}

Thanks go to Jesse Geneson (MIT) for his steadfast and generous support as a mentor; the Center for Excellence in Education, the Research Science Institute, and MIT for giving me the opportunity to conduct this research; Dr. Tanya Khovanova (MIT) for coordinating the RSI mathematics research program and for inspiring me to always ``look at the big picture"; Dr. Jake Wildstrom (University of Louisville) for his copious edits and excellent tutorship; Kartik Venkatram (MIT) for pointing out a slick binary representation of the conditions for $r(4,n)$; and Akamai Technologies and the Hipsman Family for sponsoring my stay at the Research Science Institute.

\begin{singlespace}
 


\end{singlespace}

%

\appendix


\section{Base Cases for $4\times n$ Grids}\label{app:basecase}

\begin{figure}[h]
\centering
\subfloat[7-ranking of $G_{4,4}$]{\includegraphics[height=1in]{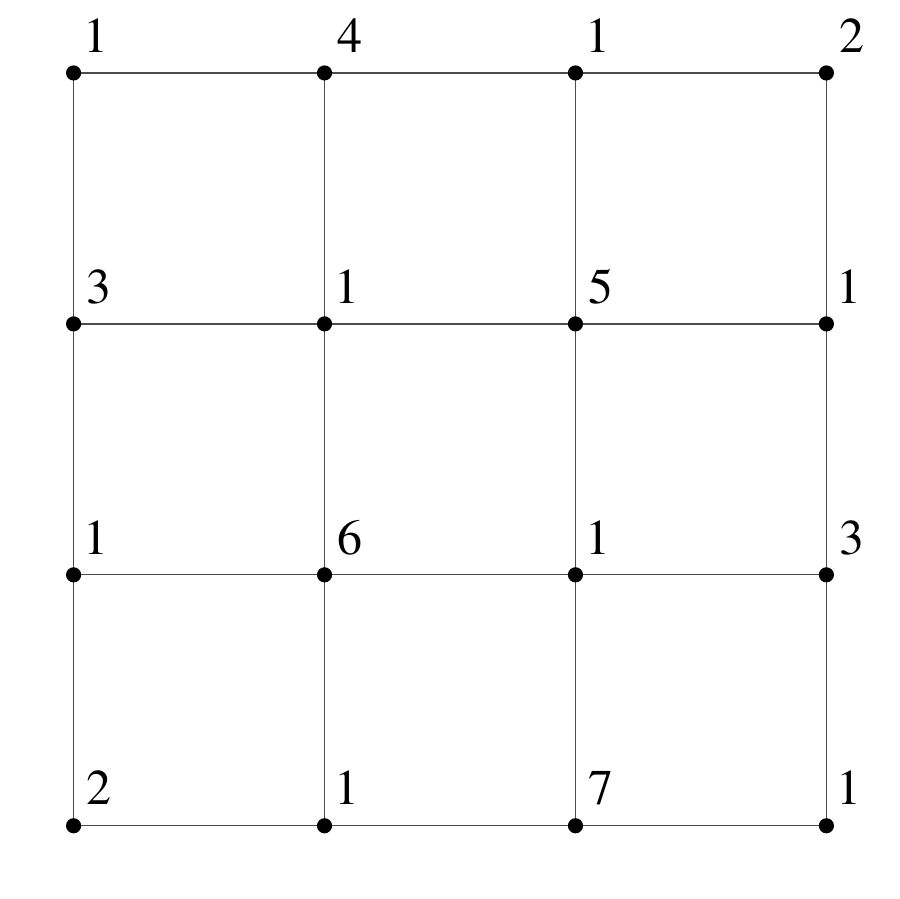}}
\subfloat[8-ranking of $G_{4,5}$]{\includegraphics[height=1in]{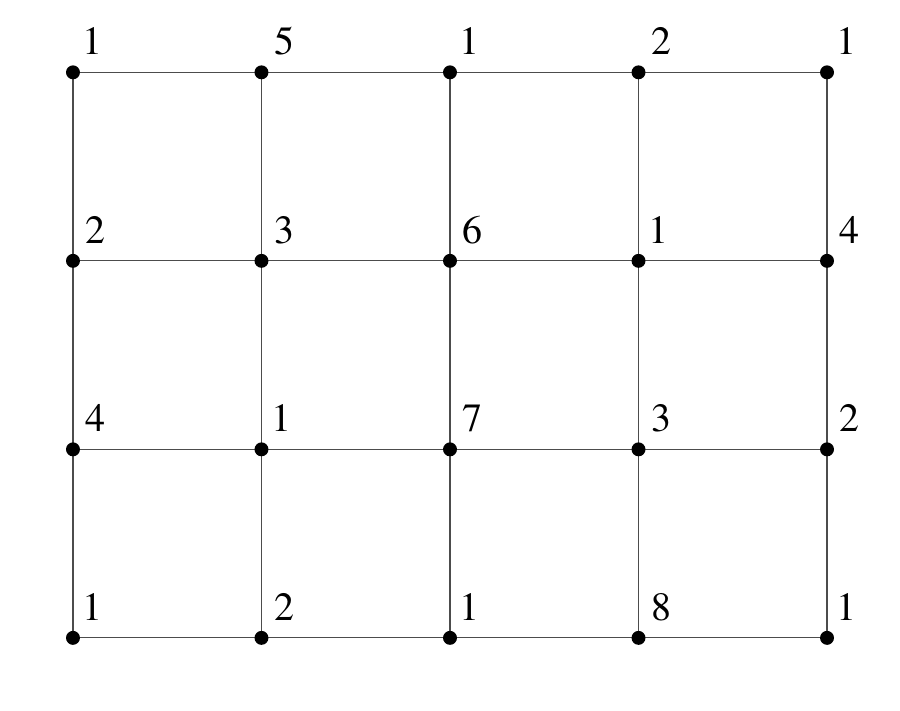}}
\subfloat[8-ranking of $G_{4,6}$]{\includegraphics[height=1in]{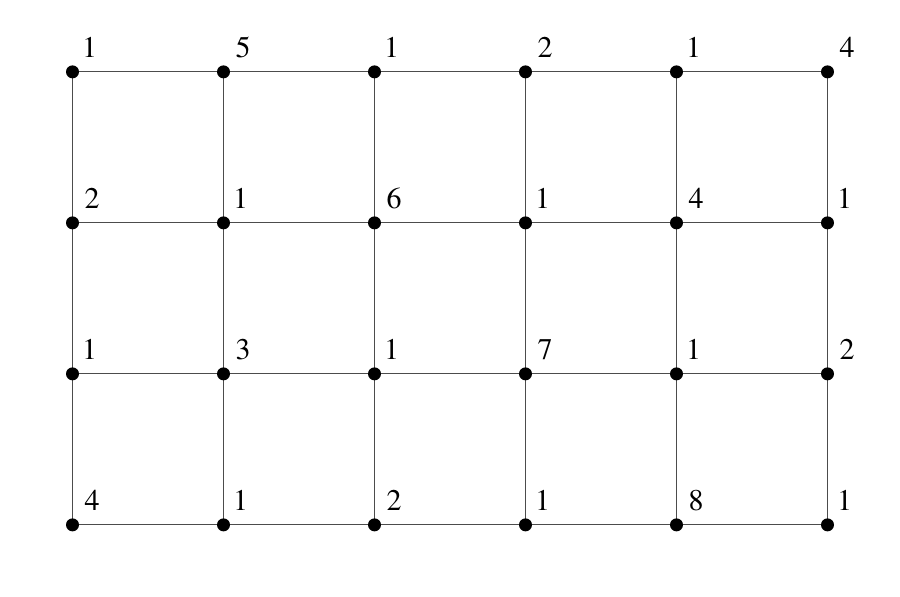}}
\\
\subfloat[9-ranking of $G_{4,7}$]{\includegraphics[height=1in]{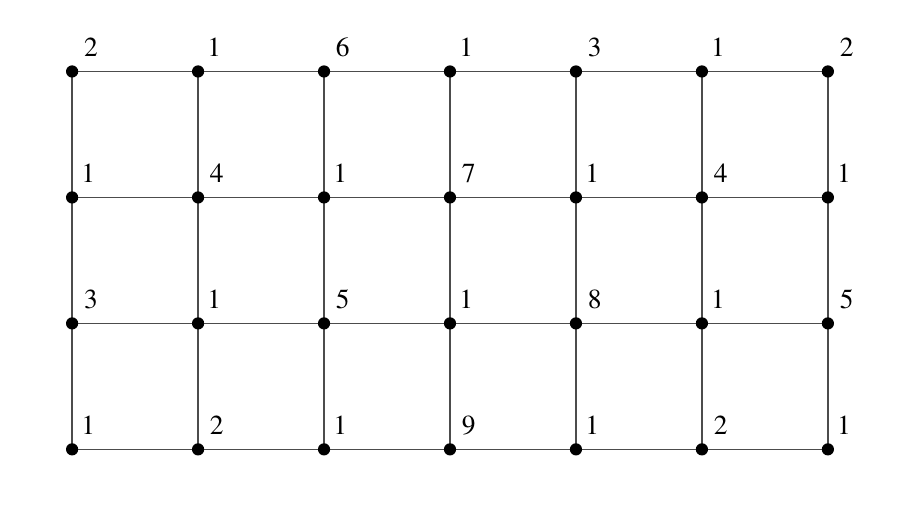}}
\subfloat[10-ranking of $G_{4,8}$]{\includegraphics[height=1in]{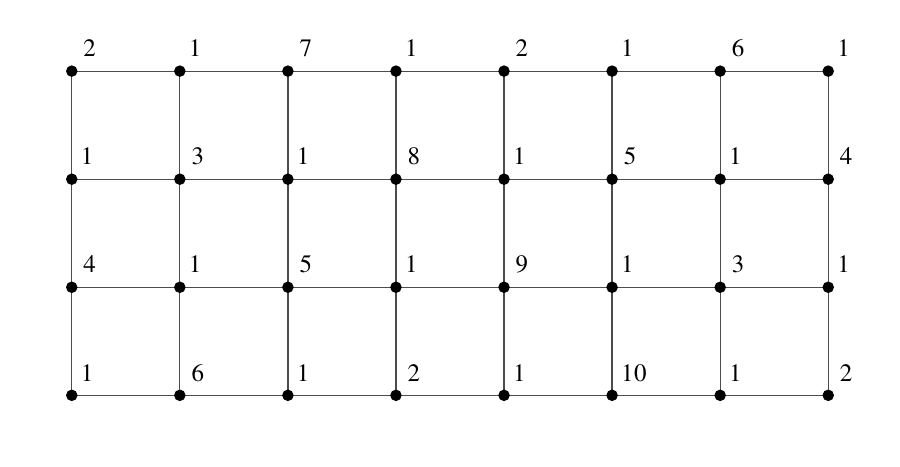}}
\caption{$r(n)$-ranking of $G_{4,n}$}
\label{fig:gridbase}
\end{figure}

In Figure~\ref{fig:gridbase}, we show by construction that $r(4)\le7$, $r(5)\le8$, $r(6)\le8$, and $r(7)\le9$.  Here, we prove that these upper bounds are indeed the actual rank numbers for these grids.

\begin{case}
$r(4)=7$
\end{case}

\begin{proof}



Assume a 6-ranking of $G_{4,4}$ exists.  If $\alpha=4$, then to separate two vertices labeled 4, the vertices 5 and 6 must be neighbors of a corner.  But then the corner vertex labeled 4 could be labeled 1, so the graph is not minimal.  If $\alpha=3$, then it is easxy to check that removing any cut set of three vertices will induce a 3-ranking of a $G_{2,4}$ subgraph, which is impossible.  If $\alpha=2$, then it is easy to check that removing any cut set of four vertices will induce a 2-ranking of a $P_5$ subgraph, which is impossible.  If $\alpha=1$, then besides the 5 vertices labeled with colors larger than 1, 11 vertices must be colored with 1, but then some pair of vertices colored 1 will be adjacent.

Thus, a 6-ranking of $G_{4,4}$ does not exist.\end{proof}

\begin{case}
$r(5)=8$
\end{case}

\begin{proof}Assume a 7-ranking of $G_{4,5}$ exists.  If $\alpha=5$, then to separate two vertices labeled 5, the vertices colored 6 and 7 must be neighbors of a corner, inducing a 5-ranking of a $G_{3,4}$ subgraph, which is impossible.  Similarly, if $\alpha=4$, then $|C(f)|\le 3$, and it is easy to check that removing any cut set of three vertices will induce a 4-ranking of either a $G_{2,5}$ or a $G_{3,3}$ subgraph.  If $\alpha=3$, then $|C(f)|\le 4$, and it is easy to check that removing any cut set of four vertices will still induce a 3-ranking of a $G_{2,3}$ subgraph, which is impossible.  If $\alpha=2$, then $|C(f)|\le 5$, and it is easy to check that removing any five vertices will still induce a 2-ranking of a $P_4$ subgraph, which is impossible.  If $\alpha=1$, then besides the 6 vertices labeled with colors larger than 1, 14 vertices must be colored with 1, but then some pair of vertices colored 1 will be adjacent.

Thus, a 7-ranking of $G_{4,5}$ does not exist.\end{proof}

\begin{case}
$r(6)=8$
\end{case}

\begin{proof}Because $P_4\times P_5$ is a subgraph of $P_4\times P_6$, $r(6)\ge r(5)=8$.  We show an 8-ranking of $P_4\times P_6$ below, so $r(6)=8$.\end{proof}

\begin{case}
$r(7)=9$
\end{case}

\begin{proof}Assume an 8-ranking of $G_{4,7}$ exists.  If $\alpha=6$, then to separate two vertices labeled 6, the two vertices labeled 7 and 8 must be neighbors of a corner.  But then the corner vertex could be labeled with a 1, and we would still have an 8-ranking, so the original 8-ranking is not minimal.  If $\alpha=5$, then $|C(f)|\le 3$, but it is easy to check that removing any three vertices in the cut set will induce a 5-ranking of either a $G_{2,7}$ or a $G_{3,5}$ subgraph, which is impossible.  If $\alpha=4$, then $|C(f)|\le 4$, but it is easy to check that removing any four vertices in the cut set will induce a 4-ranking of a $G_{3,3}$ subgraph, which is impossible.  If $\alpha=3$, then $|C(f)|\le 5$, but it is easy to check that removing any five vertices in the cut set will induce a 3-ranking of a $G_{2,3}$ subgraph, which is impossible.  If $\alpha=2$, then $|C(f)|\le 6$.  Note that each column can have at most three vertices labeled at most 2, but we have at least 7 remaining vertices to label, but only six vertices can be labeled more than 2.  If $\alpha=1$, then $|C(f)\le 7$, but after labeling the cut set vertices, we have 21 vertices to label with 1, but then some pair of vertices colored 1 will be adjacent.

Thus, an 8-ranking of a $G_{4,7}$ does not exist.\end{proof}

\begin{case}
$r(8)=10$
\end{case}

\begin{proof}Assume a 9-ranking of $G_{4,8}$ exists. If $\alpha=7$, then to separate two vertices labeled 7, the two vertices labeled 8 and 9 must be neighbors of a corner.  But then the corner vertex could be labeled with a 1, and we would still have a 9-ranking, so the original 8-ranking is not minimal.  If $\alpha=6$, then $|C(f)|\le 3$, but it is easy to check that removing any three vertices in the cut set will induce a 6-ranking of a $G_{3,6}$ subgraph and/or separate $G_{4,8}$ into two connected components, one of which is a vertex.  In the latter case, that vertex's label of 6 could be replaced with a 1 and still leave a 9-ranking, and the original ranking would not be minimal.  If $\alpha=5$, then $|C(f)|\le 4$, but it is easy to check that removing any four vertices in the cut set will induce a 5-ranking of a $G_{3,4}$ subgraph, which is impossible.  If $\alpha=4$, then $|C(f)|\le 5$, but it is easy to check that removing any five vertices in the cut set will induce a 4-ranking of a $G_{3,3}$ subgraph, which is impossible.  If $\alpha=3$, then $|C(f)|\le 6$, but it is easy to check that removing any six vertices in the cut set will induce a 3-ranking of a $G_{2,3}$ subgraph, which is impossible.  If $\alpha=2$, then $|C(f)|\le 7$, 
but as in the previous case, each column can have at most three vertices labeled at most 2, but we have at least 8 remaining vertices to label, but only seven vertices can be labeled more than 2.  If $\alpha=1$, then $|C(f)|\le 8$, but after labeling the cut set vertices, we have 24 vertices to label with 1, but then some pair of vertices colored 1 will be adjacent.

Thus, a 9-ranking of a $G_{4,8}$ does not exist.\end{proof}

\section{Base Cases for $4\times n$ Grids with Two Sticky Ends}\label{app:basesticky}

\begin{figure}[h]
\centering
\subfloat[5-ranking of $G_{4,1}$]{\includegraphics[height=1in]{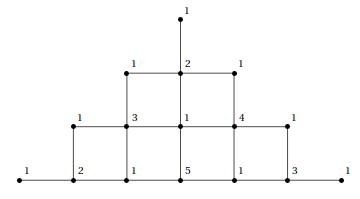}}
\subfloat[6-ranking of $G_{4,2}$]{\includegraphics[height=1in]{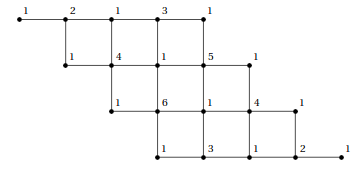}}
\\
\subfloat[7-ranking of $G_{4,3}$]{\includegraphics[height=1in]{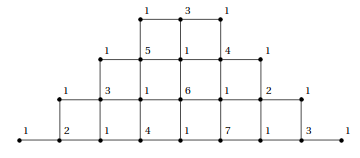}}
\subfloat[8-ranking of $G_{4,4}$]{\includegraphics[height=1in]{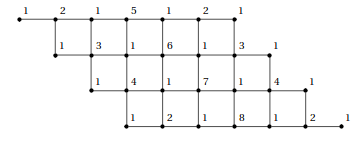}}
\caption{$r(n)$-ranking of $G_{4,n}$ with two sticky ends}
\label{fig:stickybase}
\end{figure}

Recall that we define $s(n)$ to be the rank number of a $4\times n$ grid with two sticky ends.  We show in Figure~\ref{fig:stickybase} that $s(1)\le 5$, $s(2)\le 6$, $s(3)\le 7$, and $s(4)\le 8$.

\section{Base Cases for $4\times n$ Grids with One Sticky End}\label{app:baseonesticky}

\begin{figure}[h]
 \centering
  \subfloat[6-ranking of $G_{4,3}$]{\label{fig:4-3sticky}\includegraphics[height=1in]{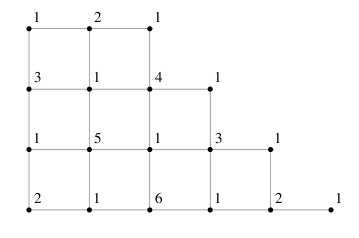}}
  \subfloat[7-ranking of $G_{4,4}$]{\label{fig:4-4sticky}\includegraphics[height=1in]{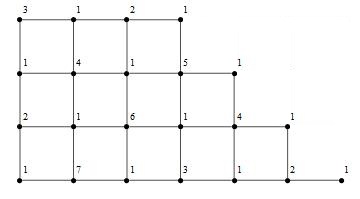}}
  \subfloat[8-ranking of $G_{4,5}$]{\label{fig:4-5sticky}\includegraphics[height=1in]{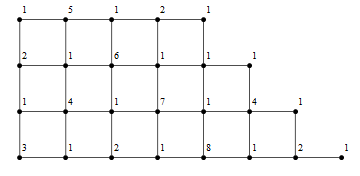}}
    \caption{$r(n)$-ranking of $G_{4,n}$ with one sticky end}
\label{fig:4sticky}
\end{figure}

\section{Base Cases for $4\times n$ Grids Missing Corners}\label{app:corner}

\begin{figure}[h]
 \centering
  \subfloat[5-ranking of $G_{4,3}$ minus two adjacent corners]{\label{fig:4-3cut1}\includegraphics[height=0.15\textwidth]{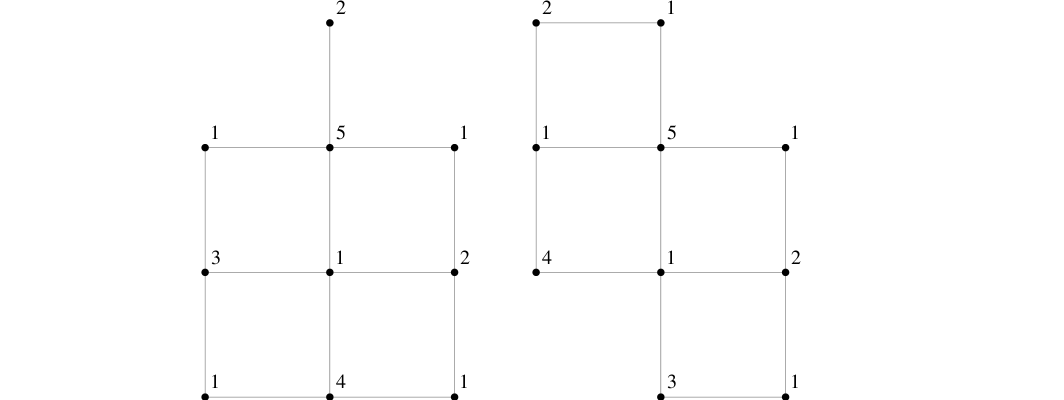}}
    \hspace{0.5cm}
  \subfloat[7-ranking of $G_{4,5}$ minus a corner]{\label{fig:4-5corner}\includegraphics[height=0.15\textwidth]{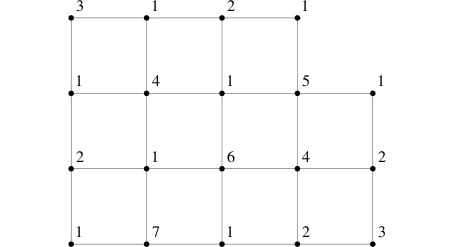}}
  \caption{$r(n)$-ranking of $G_{4,n}$ missing one or more corners}
\end{figure}

\section{Induced Subgraphs of $G_{4,13}$}\label{app:case4}

\begin{figure}[h]
 \centering
  \subfloat[]{\label{fig:13must1}{\includegraphics[width=0.26\textwidth]{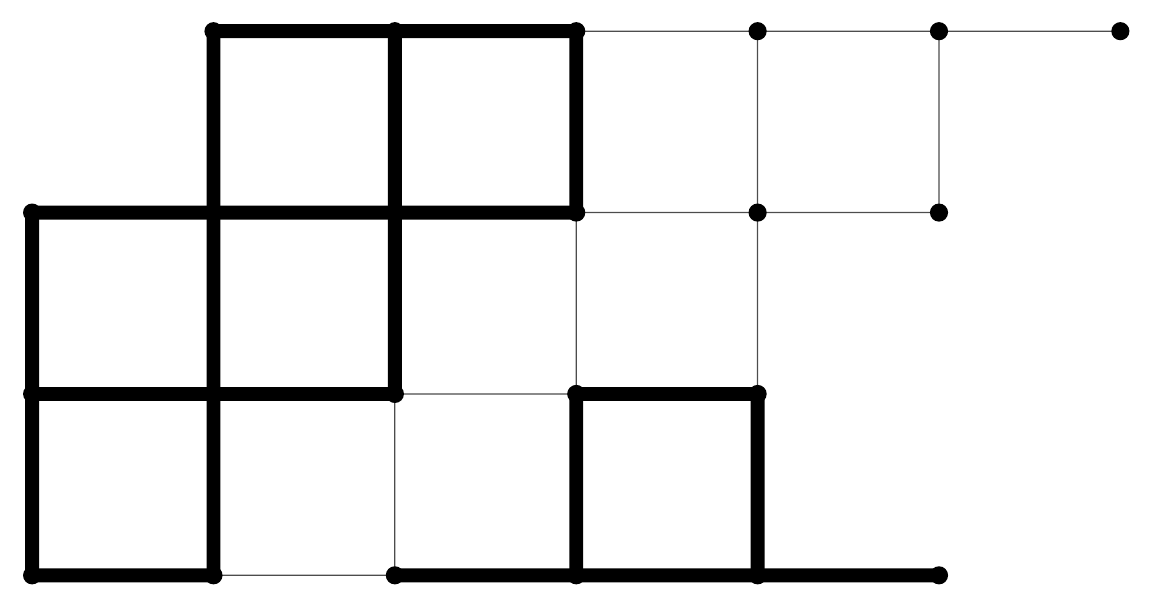}}}
  \subfloat[]{\label{fig:13must2}{\includegraphics[width=0.26\textwidth]{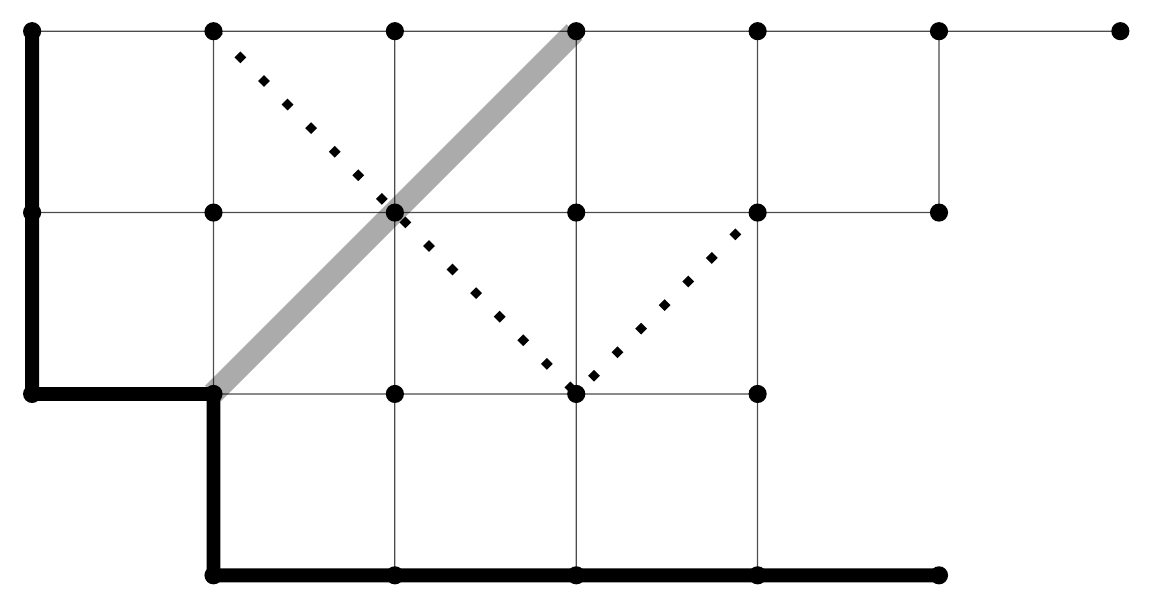}}}
  \subfloat[]{\label{fig:13must3}{\includegraphics[width=0.26\textwidth]{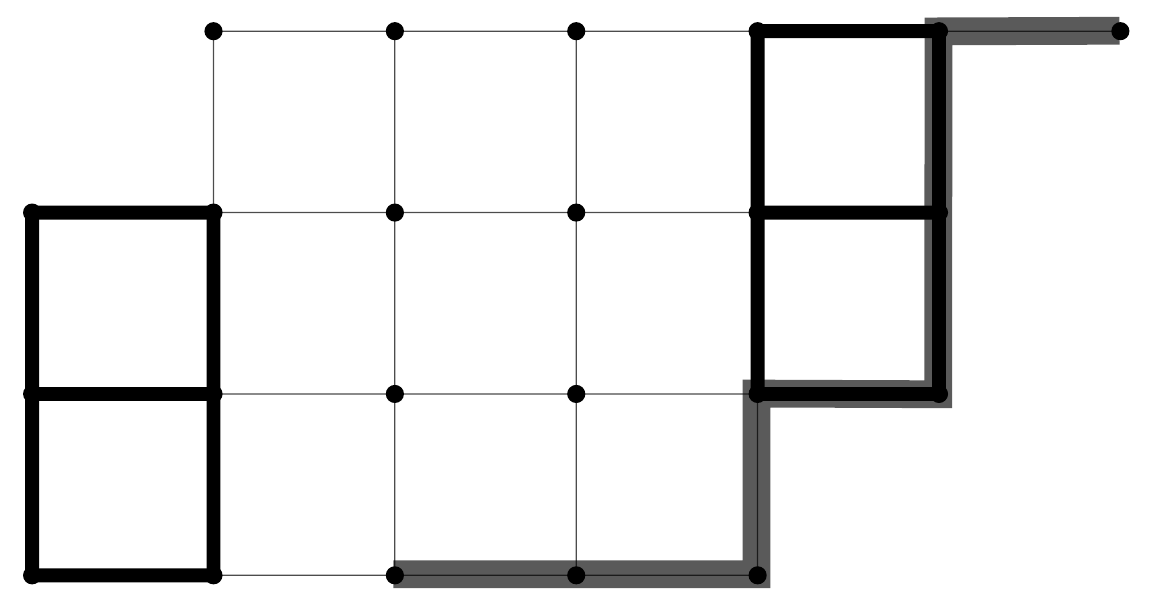}}}
  \subfloat[]{\label{fig:13must4}{\includegraphics[width=0.26\textwidth]{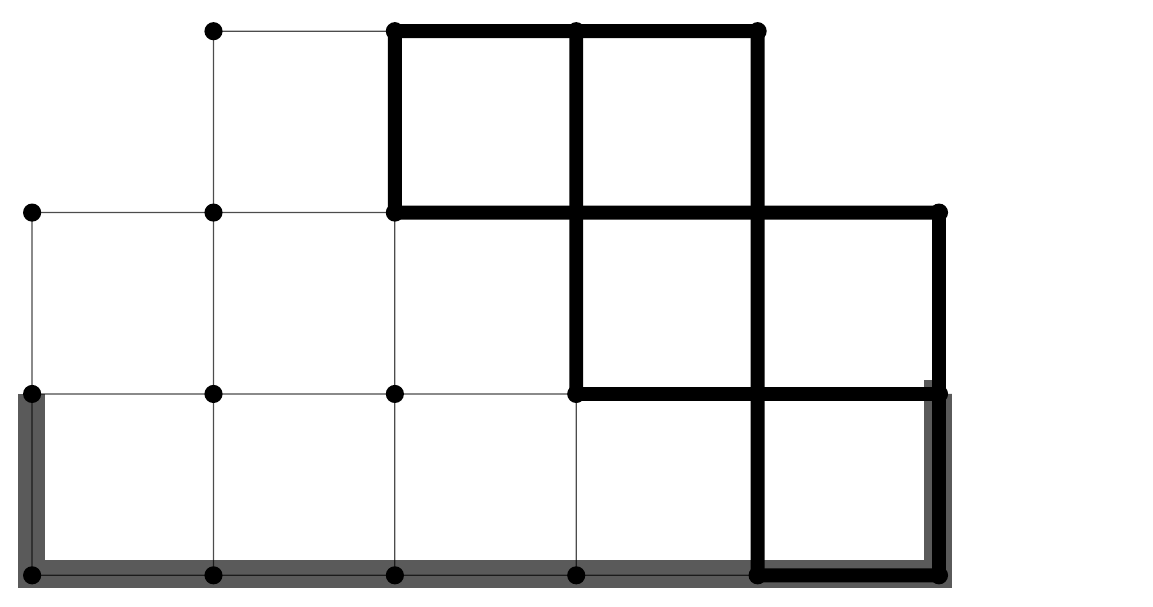}}}   
  \caption{Four graphs, one of which any induced subgraph of $G_{4,13}$ must be a subgraph}
\end{figure}

Observe that any cut set of three vertices of the graph in Figure~\ref{fig:13must1} will induce a $3\times 3$ grid, a $2\times 5$ grid, or the ``stairstep'' graph bolded in the figure, all of which have rank numbers of 5. Any cut set of four vertices will induce a $2\times 3$ grid or the $2\times 2$ grid with two additional vertices (call this a ``hill graph,'' also bolded in the figure, both of which have rank numbers of 4.

Any cut set of three vertices of the graph in Figure~\ref{fig:13must2} can induce the same subgraphs as it will for the graph in Figure~\ref{fig:13must1}, but if not, observe that the cut set must be the one shown in grey in Figure~\ref{fig:13must2}. Any cut set of four vertices can induce the same subgraphs as it will for the graph in Figure~\ref{fig:13must1}, but if not, as shown in Figure~\ref{fig:13must2}, it will induce a path graph of at least 8 vertices, which also has a rank number of 4.

Any cut set of three vertices of the graph in Figure~\ref{fig:13must3} will induce a $3\times 3$ grid or a ``stairstep'' graph. Observe that any cut set of four vertices that contains vertices on the bolded path graph of eight vertices and on both of the bolded $2\times 3$ grids must induce another $2\times 3$ grid or ``hill'' graph.

Finally, any cut set of three vertices of the graph in Figure~\ref{fig:13must4} will likewise induce a $3\times 3$ grid or a ``stairstep'' graph. Furthermore, for a cut set of four vertices to contain vertices on the bolded path graph of eight vertices, it must induce a $2\times 3$ grid.

%
%
%
%


%

\section{Explicit Upper Bound for Triangle Grids}\label{app:tri}

\begin{lem}
 For $n\ge3$, $$\text{\emph{tri}}(n)\le 2n-2\lfloor\log_2(n+1)\rfloor+1.$$
\end{lem}

\begin{proof}
 In \cite{alpert}, Alpert showed tha$k=\tri{n}\le n-1+\tri{\left\lceil \frac{n-2}{2} \right\rceil}.$  For convenience, we will use the inequality $\tri{n}\le n-1+\tri{\frac{n-1}{2}}.$  Iterating this $k$ times gives the inequality 
\begin{eqnarray*}
\tri{n} &\le& \left(\displaystyle\sum^{k}_{i=1}\frac{n-(2^{k-1}-1)}{2^{k-1}}\right)-k+\tri{\frac{n-(2^k-1)}{2^k}} \\
	&=&   (n+1)\left(2-\frac{1}{2^k}\right)-2k+\tri{\frac{n-(2^k-1)}{2^k}}
\end{eqnarray*}

We want the dimensions of the triangle to be at most $1\times 1$, so $0<\frac{n-(2^k-1)}{2^k}\le 1$ when $\frac{n+1}{2}\le 2^k<n+1\Rightarrow k=\lfloor\log_2(n+1)\rfloor$.  Then the above becomes
\begin{eqnarray*}
\tri{n}&<&(n+1)\left(2-\frac{1}{n+1}\right)-2\lfloor\log_2(n+1)\rfloor+\tri{1} \\
	&=&2n-2\lfloor\log_2(n+1)\rfloor+2
\end{eqnarray*}

Thus, $\text{tri}(n)\le 2n-2\lfloor\log_2(n+1)\rfloor+1$.
\end{proof}

\section{Comparing Upper Bounds for General Grids}\label{app:gen}
\begin{lem}
For $n\le N$, where $N=O\left(m^{\frac{3}{2}}\right)$, the bound $\chi_r(G_{m,n})\le 3m-2\log_2(m+1)+1+\chi_r\left(G_{m,\left\lceil\frac{n-m}{2}\right\rceil-1}\right)$ is better than  $\chi_r(G_{m,n})\le m+\chi_r(G_{m,\lceil\frac{n-1}{2}\rceil})$.
\end{lem}

\begin{proof}
To compare, we make both Alpert's and our bound explicit. Call the former $N_A$ and the latter $N_B$.  For simplicity, let $N_A\ge m+\chi_r(G_{m,\lfloor\frac{n-1}{2}\rfloor})$.  Iterating this $k$ times gives $N_A\ge mk+\chi_r\left(G_{m,\frac{n-(2^{k}-1)}{2^k}}\right).$  We want the dimensions of the rectangle to be at most $1\times m$, so $0<\frac{n-(2^k-1)}{2^k}\le 1$ when $\frac{n+1}{2}\le 2^k<n+1\Rightarrow \log_2(n+1)-1\le k<\log_2(n+1)$.  More specifically, $k=\lfloor\log_2(n+1)\rfloor$.  We rewrite $N_A$ to get
\begin{align*}
 N_A&\ge m\lfloor\log_2(n+1)\rfloor+\chi_r(G_{m,1})\\
&= m\lfloor\log_2(n+1)\rfloor+\lfloor\log_2(m)\rfloor+1\\
&\ge m(\log_2(n+1)-1)+\log_2m
\end{align*}
Similarly, let $N_B\le 3m-2\log_2(m+1)+1+\chi_r\left(G_{m,\frac{n-m-1}{2}}\right).$  Then iterating this $k$ times gives $$N_B\le k(3m-2\lfloor\log_2(m+1)\rfloor+1)+\chi_r\left(G_{m,\frac{n-(m+1)(2^k-1)}{2^k}}\right).$$  We want the dimensions of the remaining rectangle to be at most $m\times m$. But an $m\times m$ grid is a subgraph of an $m\times(m+2)$ grid, which can be colored with a middle diagonal cut set of $m$ vertices and two induced subgraphs $\tria{n}$.  Thus, we can remove the recursive term in at most $k+1$ iterations.  $\frac{n-(m+1)(2^k-1)}{2^k}>0$ when $k<\log_2\left(\frac{n+m+1}{m+1}\right)$.  We rewrite $N_B$ as $$N_B \le \left\lfloor\log_2\left(\frac{n+m+1}{m+1}\right)+1\right\rfloor(3m-2\lfloor\log_2(m+1)\rfloor+1)\le 3m\left(\log_2\left(\frac{2n+2m+2}{m+1}\right)\right).$$

$N_A\ge\log_2m\left(\frac{n+1}{2}\right)^m,$ while $N_B\le\log_2\left(\frac{2n+2m+2}{m+1}\right)^{3m}$.  Then
$$n\le \frac{(m+1)^{3/2}\sqrt[2m]{m}}{8\sqrt{2}}-1$$
$$\Longrightarrow\sqrt[m]{m}\left(\frac{n+1}{2}\right)\ge\left(\frac{4n+4}{m+1}\right)^3$$
$$\Longrightarrow\sqrt[m]{m}\left(\frac{n+1}{2}\right)\ge\left(\frac{4n+2}{m+1}\right)^3$$
$$\Longrightarrow\log_2m\left(\frac{n+1}{2}\right)^m\ge \log_2\left(\frac{2n+2m+2}{m+1}\right)^{3m}$$\end{proof}

\section{Existence of $\sq{\ceil{\frac{2m}{5}}-1}$ for Small $k$}\label{app:square}

\begin{lem}
 The 0th subgrid extension for a cut set of $k=\floor{\frac{m-1}{5}}+2$ columns contains $\sq{\ceil{\frac{2m}{5}}-1}$.
\end{lem}

\begin{proof}

For $m=5x+1$, $5x+2$, $5x+3$, $5x+4$, and $5x+5$, $k=\floor{\frac{m-1}{5}}+2=x+2$.

For $m=5x+1$, the 0th subgrid extension has dimensions $(2x+1)\times(2x)$.  $\ceil{\frac{2m}{5}}-1=2x$, so the 0th subgrid extension contains $\sq{\ceil{\frac{2m}{5}}-1}$.  For $m=5x+2$, the 0th subgrid extension has dimensions $(2x+1)\times(2x)$.  $\ceil{\frac{2m}{5}}-1=2x$, so the 0th subgrid extension contains $\sq{\ceil{\frac{2m}{5}}-1}$.  For $m=5x+3$, the 0th subgrid extension has dimensions $(2x+1)\times(2x+1)$.  $\ceil{\frac{2m}{5}}-1=2x+1$, so the 0th subgrid extension contains $\sq{\ceil{\frac{2m}{5}}-1}$.  For $m=5x+4$, the 0th subgrid extension has dimensions $(2x+1)\times(2x+1)$.  $\ceil{\frac{2m}{5}}-1=2x+1$, so the 0th subgrid extension contains $\sq{\ceil{\frac{2m}{5}}-1}$.  For $m=5x+5$, the 0th subgrid extension has dimensions $(2x+1)\times(2x+2)$.  $\ceil{\frac{2m}{5}}-1=2x+1$, so the 0th subgrid extension contains $\sq{\ceil{\frac{2m}{5}}-1}$.

\end{proof}
%
%
%
%

\end{document}